\documentclass[12pt,a4paper]{article}
\usepackage{latexsym,amssymb,amsmath,mathrsfs,amsthm}
\usepackage{sectsty}
\usepackage{color}
\usepackage{bm}
\usepackage[T1]{fontenc}

\definecolor{green1}{rgb}{0,0.392157,0}
\definecolor{blue1}{cmyk}{0.4,0.4,0,0.4}
\usepackage[anchorcolor=black,%
linkcolor=green1,%
colorlinks=true,%
citecolor=blue1,%
bookmarks=true,%
bookmarksnumbered=true,%
]{hyperref}
\usepackage{cite}
\usepackage{todonotes}
\usepackage[numbers]{natbib}

\allsectionsfont{\normalsize\sc\center}

\usepackage[top=3cm,bottom=3.5cm,left=2.8cm,right=2.8cm]{geometry}

\newtheorem{theorem}{Theorem}[section]

\newtheorem{corollary}[theorem]{Corollary}

\newtheorem{definition}[theorem]{Definition}

\newtheorem{lemma}[theorem]{Lemma}

\newtheorem{notation}[theorem]{Notation}
\newtheorem{proposition}[theorem]{Proposition}
\newtheorem{example}[theorem]{Example}

\newtheorem{remark}[theorem]{Remark}

\makeatletter \@addtoreset{equation}{section} \makeatother

\renewcommand{\d}{\mathrm{d}}

\newcommand{\lip}{\mathrm{Lip}}

\newcommand{\HK}{\mathsf{H\kern-3pt K}}

\begin{document}
\title{Transportation inequalities for Markov kernels and their applications}
\author{Fabrice Baudoin\footnote{Department of Mathematics, University of Connecticut, Storrs, CT
06269, USA.  Email: \texttt{fabrice.baudoin@uconn.edu}.  Research was
supported in part by National Science Foundation grant DMS-1901315.} \and
  Nathaniel Eldredge\footnote{School of Mathematical Sciences,
    University of Northern Colorado, Greeley, CO 80369, USA.  Email:
    \texttt{neldredge@unco.edu}.  Research was supported in part by
    Simons Foundation grant \#355659.}}
\date{March 22, 2021}

\maketitle

\begin{abstract}
  We study the relationship between functional inequalities for a
  Markov kernel on a metric space $X$ and inequalities of
  transportation distances on the space of probability measures
  $\mathcal{P}(X)$.  Extending results of Luise and Savar\'e on
  Hellinger--Kantorovich contraction inequalities for the particular
  case of the heat semigroup on an $RCD(K,\infty)$ metric space, we
  show that more generally, such contraction inequalities are
  equivalent to reverse Poincar\'e inequalities.  We also adapt the
  ``dynamic dual'' formulation of the Hellinger--Kantorovich distance
  to define a new family of divergences on $\mathcal{P}(X)$ which
  generalize the R\'enyi divergence, and we show that contraction
  inequalities for these divergences are equivalent to the reverse
  logarithmic Sobolev and Wang Harnack inequalities.  We discuss applications
  including results on the convergence of Markov processes to
  equilibrium, and on quasi-invariance of heat kernel measures in
  finite and infinite-dimensional groups.
\end{abstract}

\tableofcontents

\section{Introduction}

The goal of this paper is to build upon recent results of G.~Luise and
G.~Savar\'e \cite{luise-savare} on contraction properties of the flow
of a heat semigroup in spaces of measures.  There, the authors study a
``dynamic dual'' formulation of various distances between probability
measures on a metric measure space, including the Kantorovich--Wasserstein and
Hellinger distances as well as a family of Hellinger--Kantorovich
distances $\HK_\alpha$ introduced in \cite{LMS2018}.  They focus on
the setting of $\mathrm{RCD}(K, \infty)$ spaces, in which the
canonical heat semigroup $P_t$ generated by the Cheeger energy
satisfies a Bakry--\'Emery curvature condition; these spaces are, very
roughly speaking, more general analogues of Riemannian manifolds with
Ricci curvature bounded from below.  Under this assumption, they
obtain contraction inequalities of the form
\begin{equation}\label{LS-contract-intro}
  \mathsf{He}_2(\mu_0 P_t, \mu_1 P_t) \le \HK_{\alpha(t)}(\mu_0, \mu_1) 
\end{equation}
where $\mu P_t$ denotes the dual action of the heat semigroup $P_t$ on
the probability measure $\mu$, $\mathsf{He}$ and $\HK$ are the Hellinger and
Hellinger--Kantorovich distances respectively, and $\alpha(t)$ depends
on $K$.  The proof is based on the fact that $\mathrm{RCD}(K, \infty)$
spaces satisfy a reverse Poincar\'e inequality of the form 
\begin{equation}\label{rev-poincare-RCD-intro}
  |\nabla P_t f|^2 \le \frac{K}{e^{2Kt}-1} (P_t(f^2) - (P_t f)^2).
\end{equation}
Indeed, the inequality \eqref{rev-poincare-RCD-intro}, with its specific
form of the time-dependent constant $\frac{K}{e^{2Kt}-1}$, is one of
many functional inequalities that are equivalent to the Bakry--\'Emery
curvature condition; see for instance \cite[Proposition
  3.3]{bakry-tata}.

The first goal of the present paper is to further study the
relationship between reverse Poincar\'e inequalities and
Hellinger--Kantorovich contraction inequalities.  Our first main
result is Theorem \ref{equi-func}, in which we show that the
implication between the two holds in a much more general setting than
$\mathrm{RCD}(K, \infty)$ spaces.  We suppose only that we have a
Markov operator $P$ acting on a sufficiently nice metric space $X$,
and we show that if $P$ satisfies a reverse Poincar\'e inequality of
the form
\begin{equation}\label{rev-poincare-intro}
  |\nabla Pf|^2 \le C(P(f^2) - (Pf)^2), \qquad f \in \lip_b(X)
\end{equation}
then we obtain a Hellinger--Kantorovich contraction of the form
\begin{equation}\label{HK-contract-intro}
  \mathsf{He}_2(\mu_0 P, \mu_1 P) \le \HK_{4/C}(\mu_0, \mu_1)
\end{equation}
for all probability measures $\mu_0, \mu_1$ on $X$.  In particular,
\eqref{HK-contract-intro} holds in non-RCD models where there is a
semigroup $P_t$ which satisfies \eqref{rev-poincare-intro} for each
$t$, but with a time-dependent constant $C(t)$ that is not of the form
appearing in \eqref{rev-poincare-RCD-intro}.  We discuss several
examples and applications in Section \ref{equilibrium-sec}, including
subelliptic diffusions arising in sub-Riemannian geometry,
non-symmetric Ornstein--Uhlenbeck operators on Carnot groups, Langevin
dynamics driven by L\'evy processes, and others.

Furthermore, in this general setting, we are able to show (also in
Theorem \ref{equi-func}) that the converse implication holds as well,
so that \eqref{rev-poincare-intro} and \eqref{HK-contract-intro} are
in fact equivalent.  They are also equivalent to a parabolic Harnack
inequality.  Thus the Hellinger--Kantorovich contraction can be seen
as a new aspect of a well-known family of functional inequalities,
providing additional tools and motivations for their study.

The key tool in all of this is the dynamic dual formulation of the
Hellinger--Kantorovich distance (Definition \ref{W-def}), originally
introduced in \cite{LMS2018}, which expresses $\HK_a(\mu_0, \mu_1)$ as
the supremum of $\int \varphi_1\,\d\mu_1 - \int \varphi_0\,\d\mu_0$
over a family of time-dependent functions $\varphi_s : [0,1] \times X
\to \mathbb{R}$ satisfying a certain Hamilton--Jacobi partial
differential inequality in time and space.  This formula extends the
so-called Kantorovich duality for the Kantorovich--Wasserstein
distance, and also includes an expression for the Hellinger distance.
Having the distance defined in terms of solutions of a partial
differential inequality makes it particularly convenient to relate it
to functional inequalities where the gradient appears, as we
demonstrate in Theorems \ref{poincare-func-ineq} and \ref{equi-func}.
Indeed, this technique also provides an extension of the Kuwada
duality theorem \cite{kuwada-duality, LMS2018}, relating $L^2$
gradient estimates to a Kantorovich--Wasserstein contraction
inequality; see Corollary \ref{gradient-wasserstein-contraction}.

Pursuing this idea in a different direction, in Section
\ref{renyi-sec} we use a dynamic dual approach to formulate a new
family of transportation-cost divergences $T_{a,b}$ on the space of
probability measures, which are of ``entropic'' type and include the
R\'enyi divergence.  In place of reverse Poincar\'e inequalities, this
family is designed to connect with reverse logarithmic Sobolev
inequalities of the form
\begin{equation}\label{rlsi-intro}
  Pf |\nabla \ln Pf|^2 \le C(P(f \ln f) - (Pf) \ln Pf).
\end{equation}
We note that an equality of the type \eqref{rlsi-intro} is stronger
than an inequality of the type \eqref{rev-poincare-intro}, by using
\eqref{rlsi-intro} with $1+\varepsilon f$.  After developing some
properties of the $T_{a,b}$ divergences, we show in Theorem
\ref{entropic-equivalence} that \eqref{rlsi-intro} is actually
equivalent to a family of entropic transportation-cost contraction
inequalities for $T_{a,b}$, of the form
\begin{equation}\label{T-contract-intro}
  T_{0, \kappa C}(\mu_0 P, \mu_1 P) \le T_{\kappa, \kappa C}(\mu_0,
  \mu_1), \qquad \kappa > 0.
\end{equation}
These two statements are moreover equivalent to a Wang-type parabolic
Harnack inequality, as well as to an integrated Harnack inequality
(see Remark \ref{integrated-harnack}).  Thus, the reverse log Sobolev
inequality \eqref{rlsi-intro} also has a ``transport'' aspect.  We
discuss in Section \ref{quasi-sec} how \eqref{T-contract-intro} can be
used, in finite or infinite dimensions, to prove certain
quasi-invariance results that were previously obtained via Wang
Harnack inequalities or other methods; see for instance \cite{bgm}.

For other applications, and a general overview of reverse Poincar\'e
and log-Sobolev inequalities and of the connections with Harnack type
inequalities, we refer to the book \cite{MR3099948}.

\section{General setup and notation}\label{setup-sec}

Throughout the paper, unless otherwise specified, $(X,d)$ denotes a
complete, proper, separable metric space which is a length space; in
particular, path connected.  We suppose that $X$ is equipped with a strong upper gradient
$|\nabla f|$ as defined in \cite[Definition 1.2.1]{AGS-Book-2008}.
More precisely, for a measurable function $f: X \to \mathbb{R}$ we
define
\begin{equation}
|\nabla f|(x) =\lim_{r \to 0} \sup_{0<d(x,y) \le r} \frac{ |
  f(x)-f(y)|}{d(x,y)}
\end{equation}
and denote by $\lip_b(X)$ the space of all bounded Lipschitz functions
on $X$. Then, we have the following result:

\begin{lemma}[Proposition 1.11, \cite{Cheeger}]
For every $f \in \lip_b(X)$, $|\nabla f|$ is a strong upper gradient
in the sense that for each rectifiable curve $\gamma: [0, L] \to X$
parametrized by arc-length we have
\[
| f( \gamma(L)) -f (\gamma(0))| \le \int_0^L | \nabla f | ( \gamma(s))
\d s.
\]
\end{lemma}

One may also verify that $|\nabla f|$ satisfies the chain rule:

\begin{lemma}\label{chain-rule}
  If $f : X \to \mathbb{R}$ is Lipschitz in a neighborhood of $x$ and $\phi :
  \mathbb{R} \to \mathbb{R}$ is differentiable at $f(x)$, then
  $|\nabla (\phi \circ f)|(x) = |\phi'(f(x))|\, |\nabla f|(x)$.
\end{lemma}

Let $\mathcal{B}_X$ denote
the Borel $\sigma$-algebra of $(X,d)$, and $\mathcal{P}(X)$ the set of
Borel probability measures on $X$.  We suppose we are given a
Markov probability kernel $P : X \times \mathcal{B}_X \to [0,1]$, and
we denote by $Pf$, $\mu P$ the usual action of $P$ on bounded Borel
functions $f$ and Borel probability measures $\mu$, i.e.
\begin{equation*}
  Pf(x) := \int_X f(y) P(x,\d y), \quad \mu P(A) := \int_X P(x,A)
  \mu(\d x).
\end{equation*}

In some applications, $P$ will be taken to be a Markov semigroup
$P_t$, which may or may not be symmetric with respect to some
reference measure.  Our setting is similar to
\cite{kuwada-duality}. This is more general than the setting of
\cite{luise-savare}, which only considered the symmetric semigroup
$P_t$ generated by the Cheeger energy with respect to the given
gradient and a given reference measure.

Given $\mu_0, \mu_1 \in \mathcal{P}(X)$, the $2$-Kantorovich--Wasserstein distance
$W_2(\mu_0, \mu_1)$ is defined as usual by
\begin{equation}
  W_2(\mu_0, \mu_1)^2 := \inf_\pi \int d(x_0,
  x_1)^2\,\mu(\d x_0, \d x_1),
\end{equation} the infimum taken over all couplings
$\pi \in \mathcal{P}(X \times X)$ of $\mu_0, \mu_1$.  In particular,
for point masses $\mu_i = \delta_{x_i}$, we have $W_2(\delta_{x_0},
\delta_{x_1}) = d(x_0, x_1)$.  We let $\mathcal{P}_2(X) \subset
\mathcal{P}(X)$ denote the Wasserstein space of probability measures
$\mu$ having a finite second moment, i.e. for which $\int_X d(x, x_0)^2\,\mu(\d
x) < \infty$ for some (equivalently, all) $x_0 \in X$.

The
$2$-Hellinger distance is defined by 
\begin{equation}
  \mathsf{He}_2(\mu_0, \mu_1)^2 := \int_X \left(\sqrt{\frac{\d \mu_1}{\d m}} - \sqrt{\frac{\d
    \mu_0}{\d m}}\right)^2\,\d m
\end{equation}
where $m$ is any measure such that $\mu_0, \mu_1$ are both absolutely
continuous with respect to $m$; the definition is independent of $m$.  Convergence
in Hellinger distance is equivalent to convergence in total variation,
and we have $\mathsf{He}_2(\mu_0, \mu_1)^2 \le 2$ for all $\mu_0, \mu_1
\in \mathcal{P}(X)$, with equality iff $\mu_0, \mu_1$
are mutually singular.

The stated hypotheses on the space $X$ are meant to strike a balance
between generality and convenience; one may certainly be able to
weaken them in various ways.  We have preferred to keep the emphasis
on the techniques and their applications, rather than on stating the
most general abstract theorems.  In particular, in Section
\ref{quasi-sec} we already depart from this setting to consider
infinite-dimensional examples based on abstract Wiener space, where
$X$ is a separable Banach space (which is not proper), the test
functions are taken to be the cylinder functions instead of all
bounded Lipschitz functions, and the gradient $\nabla$ is derived from
the Malliavin gradient, whose norm is not an upper gradient with
respect to the norm distance on $X$.  This requires only trivial
modifications to the arguments in the earlier sections; we discuss the
details in Section \ref{quasi-sec}.

\section{Hellinger--Kantorovich distances and functional inequalities}\label{HK-sec}

\subsection{The dynamic dual formulation and basic properties}

In this section, we consider the family of Hellinger--Kantorovich
distances studied in \cite{LMS2016,LMS2018,luise-savare}.  We focus on
the so-called dynamic dual formulation of these distances, in which
they may be defined as the supremum of a difference of integrals over
a class of subsolutions of a Hamilton--Jacobi-type equation in time
and space variables.  This idea is directly descended from a dynamic
dual formulation of the Kantorovich--Wasserstein distance, introduced in
\cite{otto-villani-2000}.  Using this formulation of these distances,
we will see that Poincar\'e and reverse Poincar\'e type inequalities for
$P$ lead directly to contraction results with respect to these
distances (Theorem \ref{poincare-func-ineq}).

We study the Hellinger--Kantorovich distance via a slightly different
parametrization which is more convenient for our purposes.  As above,
let $\lip_b(X)$ denote the Banach space of all bounded Lipschitz
functions on $X$.  We remark for later use that for any finite measure
$\mu$ on $X$, we have $\lip_b(X)$ dense in $L^1(\mu)$, and in
particular that for any bounded Borel function $f$ there is a sequence
$f_n \in \lip_b(X)$ with $f_n \to f$ $\mu$-a.e. and boundedly.

\begin{definition}\label{W-def}
  Let $a,b \ge 0$.  We denote by $\mathcal{A}_{a,b}$ the class of all
  functions $\varphi = \varphi_s(x) \in C^1([0,1], \lip_b(X))$ satisfying the
  differential inequality
  \begin{equation}\label{W-diff-ineq}
    \partial_s \varphi_s + a |\nabla \varphi_s|^2 + b \varphi_s^2 \le 0.
  \end{equation}
  Then for probability measures $\mu_1, \mu_2 \in \mathcal{P}(X)$ we set
  \begin{equation}
    W_{a,b}(\mu_0, \mu_1) = \sup_{\varphi \in \mathcal{A}_{a,b}}
      \left[\int_X \varphi_1\,\d\mu_1 - \int_X \varphi_0\,\d\mu_0\right].
  \end{equation}
\end{definition}

To avoid confusion, we note that $W_{a,b}$ itself is not a distance on
(a subset of) $\mathcal{P}(X)$, but rather the square of a distance.

\begin{lemma}\label{W-lemma} The squared distances $W_{a,b}$ satisfy the
  following basic properties:
  \begin{enumerate}
  \item \label{W-monotone} If $a \le a'$ and $b \le b'$ then $W_{a', b'} \le W_{a,
    b}$.
  \item \label{W-scaling} For any $c > 0$, we have $W_{ca, cb} = c^{-1} W_{a,b}$.
  \item \label{W-HK} When $b > 0$, we have $W_{a,b} = b^{-1} \HK^2_{4a/b}$, where $\HK$ is
    the Hellinger--Kantorovich distance as defined in \cite[Definition
      2.11]{luise-savare}.
  \item \label{W-Wasserstein} $W_{1/2, 0} = \frac{1}{2} W_2^2$, where $W_2$ is the
    Kantorovich--Wasserstein $2$-distance.
  \item \label{W-He} $W_{0,1} = \mathsf{He}_2^2$, where $\mathsf{He}_2$ is the
    Hellinger 2-distance.
  \end{enumerate}
\end{lemma}

\begin{proof}
  Item \ref{W-monotone} is clear because when $a \le a'$ and $b \le
  b'$, we have $\mathcal{A}_{a', b'} \subseteq \mathcal{A}_{a,b}$.
  Item \ref{W-scaling} holds because $\varphi \in \mathcal{A}_{ca,
    cb}$ if and only if $c\varphi \in \mathcal{A}_{a,b}$.  For item
  \ref{W-HK}, in the notation of \cite[Eq. (39)]{luise-savare} (see
  also \cite[Section 8.4]{LMS2018}), we have $\HK^2_\alpha =
  W_{\alpha/4, 1}$, and the general statement follows using item
  \ref{W-scaling}.  Item \ref{W-Wasserstein} can be found as
  Proposition 2.10 of \cite{luise-savare}, but goes back at least as
  far as \cite[Section 3]{otto-villani-2000}; see also other
  references in \cite{luise-savare}.

  Item \ref{W-He} is almost Proposition 2.8 of \cite{luise-savare},
  but there is a slight difference because our class of functions
  $\mathcal{A}_{0,1}$ is required to be Lipschitz in space, whereas
  \cite[Eq.~(32)]{luise-savare} uses functions which are only bounded.
  This is easily handled with a straightforward approximation
  argument, which we now give for completeness.
  
  Let
  $\mathcal{A}_{0,1}^B = \{ \varphi \in C^1([0,1], B(X)) : \partial_s
  \varphi_s + \varphi_s^2 \le 0\}$.  The statement of
  \cite[Proposition 2.8]{luise-savare} is that
  \begin{equation*}
    \mathsf{He}_2^2(\mu_0, \mu_1) = \sup_{\psi \in
      \mathcal{A}_{0,1}^B} \int \psi_1\,\d\mu_1 - \int \psi_0\,\d\mu_0.
  \end{equation*}
  It is clear that $W_{0,1}(\mu_0, \mu_1) \le \mathsf{He}_2^2(\mu_0,
  \mu_1)$, since $\mathcal{A}_{0,1} \subset \mathcal{A}_{0,1}^B$.  Now
  given $\varphi \in \mathcal{A}_{0,1}$, it is clear from a
  Gr\"onwall-type argument that we must have $\varphi_s \le
  \varphi_0/(1+s\varphi_0)$ for all $s$; in particular we must have
  $\varphi_0 > -1$ or else $\varphi$ will have a singularity.  Hence
  \begin{align*}
    W_{0,1}(\mu_0, \mu_1) &= \sup\left\{\int \frac{f}{1+f}\,\d\mu_1 - \int
    f\,\d\mu_0 : f \in \lip_b(X), f > -1\right\} \\
    \intertext{and likewise}
    \mathsf{He}^2_2(\mu_0, \mu_1) &= \sup\left\{\int \frac{f}{1+f}\,\d\mu_1 - \int
    f\,\d\mu_0 : f \in B_b(X), f > -1\right\}.
  \end{align*}
  Now the result follows by noting that for each $f \in B_b(X)$ with
  $f > -1$, we can find a sequence of bounded Lipschitz functions
  $f_n$ with $f_n > -1$ and $f_n \to f$ boundedly and $(\mu_0 +
  \mu_1)$-almost everywhere.  We then have $\int f_n\,\d\mu_0 \to \int
  f\,\d\mu_0$, and since the sequence $f_n/(1+f_n)$ is
  bounded above by $1$, Fatou's lemma also gives $\limsup_{n \to
    \infty} \int \frac{f_n}{1+f_n}\,\d\mu_1 \ge \int
  \frac{f}{1+f}\,\d\mu_1$.  From this we conclude that $W_{0,1}(\mu_0,
  \mu_1) \ge \mathsf{He}_2^2(\mu_0, \mu_1)$.
\end{proof}

Thus, the (squared) distances $W_{a,b}$ naturally interpolate between the
Kantorovich--Wasserstein distance, which is perhaps the most familiar transportation
distance, and the Hellinger distance, which metrizes convergence in
total variation.  As will be seen in the next subsection, this makes
it valuable for obtaining inequalities relating these two distances.

\begin{proposition}
  \label{W-Dirac-distance-prop} If $x_0, x_1 \in X$ and $\delta_{x_0}, \delta_{x_1} \in
    \mathcal{P}(X)$ are the corresponding Dirac measures, then
    \begin{equation*}
      W_{a,b}(\delta_{x_0}, \delta_{x_1}) = 
\frac{1}{b}\left(
2 - 2 \cos
\left(\frac{\sqrt{b}}{2 \sqrt{a}}
d(x_0, x_1) \wedge \frac{\pi}{2}\right)
\right) \le \frac{1}{4a} d(x_0, x_1)^2 \wedge \frac{2}{b}.
    \end{equation*}
\end{proposition}

\begin{proof}
  For $a=\frac{1}{2}$, $b=2$, this is \cite[Eq.~(6.31)]{LMS2018}; see
  also \cite[Section 8]{LMS2018} for the explanation that the
  $\mathsf{L\kern-1pt E\kern-1pt T}$ distance corresponds to $\HK^2$,
  which is our $W_{1/2, 2}$.  Other values of $a$ can be handled by
  rescaling the distance $d$, and general values of $a,b$ are then
  covered by Lemma \ref{W-lemma} \ref{W-scaling}.

  We note, however, that the upper bound $W_{a,b}(\delta_{x_0},
  \delta_{x_1}) \le \frac{1}{4a} d(x_0, x_1)^2 \wedge \frac{2}{b}$ can
  be shown much more easily, and is comparable to the exact expression
  up to a universal constant multiple (whose value is something like
  $1.2$).  The upper bound $W_{a,b}(\mu_0,
  \mu_1) \le \frac{2}{b}$ is
  essentially trivial, and can be seen, for instance, by noting
  \begin{equation*}
    W_{a,b} \le W_{0,b} = \frac{1}{b} W_{0,1} = \frac{1}{b} \mathsf{He}_2^2
  \end{equation*}
  and that $\mathsf{He}_2^2(\mu_0, \mu_1) \le 2$ for all $\mu_0,
  \mu_1$.  The upper bound $W_{a,b}(\delta_{x_0}, \delta_{x_1}) \le
  \frac{1}{4a} d(x_0, x_1)^2$ can be seen in a similar way by
  comparing to the Kantorovich--Wasserstein distance $W_{1/2,0}$.  But
  it can also be shown directly from the ``dynamic dual'' definition
  of $W_{a,b}$.  We give the argument here, partly for comparison with
  Proposition \ref{T-point-mass-improved} below.
  
  Let $a>0$ and $b \ge 0$.    Recall that $(X,d)$ is assumed to be a
  complete length space, so there exists
  a constant speed geodesic $\gamma : [0,1] \to X$ joining $x_0$ to
  $x_1$: namely, $\gamma_0 = x_0$, $\gamma_1 = x_1$, and
  $d(\gamma_s, \gamma_t) = |s-t|d(x_0, x_1)$.  Since $\nabla$ is a
  strong upper gradient, for any Lipschitz $f : X \to \mathbb{R}$ we
  have that $f \circ \gamma$ is absolutely continuous and $\left|\frac{d}{ds}
  f(\gamma_s)\right| \le |\nabla f|(\gamma_s) d(x_0,
  x_1)$; see \cite[Definition 1.2.1]{AGS-Book-2008}.
  Now using the
  chain rule, we have
    \begin{align*}
    \varphi_1(x_1) - \varphi_0(x_0) &= \int_0^1 \frac{\d}{\d s}
    \varphi_s(\gamma_s)\,\d s \\
    &\le \int_0^1 \left[ \partial_s \varphi_s(\gamma_s) + |\nabla
      \varphi_s|(\gamma_s) d(x_0, x_1) \right]\,\d s \\
    &\le \int_0^1 \left[-a|\nabla
      \varphi_s|(\gamma_s)^2 - b \varphi_s(\gamma_s)^2 + |\nabla
    \varphi_s|(\gamma_s) d(x_0, x_1)\right]\,\d s \\
    &=  \int_0^1 \left[ -a\left(|\nabla \varphi_s(\gamma_s)| -
      \frac{1}{2a} d(x_0, x_1)\right)^2 + \frac{1}{4a} d(x_0, x_1)^2
      - b \varphi_s(\gamma_s)^2 \right]\,\d s
  \end{align*}
    by completing the square.  Discarding the two negative terms
    and taking the supremum over $\varphi_s \in \mathcal{A}_{a,b}$, we
    recover the desired bound.
\end{proof}

\subsection{Functional inequalities}

Thanks to the form of the dynamic dual definition for $W_{a,b}$, one
obtains a direct implication between functional inequalities involving
the gradient and contractions of Hellinger--Kantorovich distances.
This was the key idea in the results of \cite{luise-savare}; here we
make the implication more explicit and collect several cases into a
single statement.

\begin{theorem}\label{poincare-func-ineq}
  Suppose that for some $a>0$ and $b,\gamma,\delta \ge 0$, the Markov operator
  $P$ satisfies the functional inequality
  \begin{equation}\label{poincare-type}
    a |\nabla Pf|^2 + b (Pf)^2 \le \gamma P |\nabla f|^2 + \delta
    P(f^2), \qquad f \in \lip_b(X).
  \end{equation}
  Then we have the transportation distance contraction
  \begin{equation}
    W_{\gamma,\delta}(\mu_0 P, \mu_1 P) \le W_{a,b}(\mu_0,
    \mu_1),\qquad \mu_0, \mu_1 \in \mathcal{P}(X).
  \end{equation}
\end{theorem}

\begin{proof}
  Since $a > 0$, \eqref{poincare-type}
  implies that the Markovian operator $P$ is a bounded operator on
  $\lip_b(X)$.  Now let $\varphi \in
  \mathcal{A}_{\gamma, \delta}$.  Since $\varphi \in C^1([0,1],
  \lip_b(X))$, we have $P \varphi_s \in C^1([0,1], \lip_b(X))$ as
  well, and $P \partial_s \varphi_s = \partial_s P \varphi_s$.  Hence
  \begin{align*}
    \partial_s P \varphi_s = P \partial_s \varphi_s &\le P\left[-
      \gamma |\nabla \varphi_s|^2 - \delta \varphi_s^2 \right] \\ &=
    -\gamma P |\nabla \varphi_s|^2 - \delta P(\varphi_s^2) \\ &\le -a
    |\nabla P\varphi_s|^2 - b (P\varphi_s)^2
  \end{align*}
  where we used the fact that $P$ is positivity preserving, and the
  assumed inequality \eqref{poincare-type}.  This shows that $P
  \varphi_s \in \mathcal{A}_{a,b}$.  Thus for $\mu_0, \mu_1 \in
  \mathcal{P}(X)$ we have
  \begin{align*}
    W_{\gamma, \delta}(\mu_0 P, \mu_1 P) &= \sup_{\varphi \in
      \mathcal{A}_{\gamma, \delta}} \int_X P \varphi_1\,\d\mu_1 -
    \int_X P \varphi_0\,\d\mu_0 \\
    &\le \sup_{\psi \in \mathcal{A}_{a,b}} \int_X \psi_1\,\d\mu_1 -
      \int_X \psi_0\,\d\mu_0 \\
      &= W_{a,b}(\mu_0, \mu_1)
  \end{align*}
  as desired.
\end{proof}

\begin{corollary}\label{gradient-wasserstein-contraction}
  If $P$ satisfies the gradient estimate $|\nabla Pf|^2 \le C P|\nabla
  f|^2$ for some $C$, then for any $b \ge 0$ we have
  \begin{equation*}
    W_{1,b}(\mu_0 P, \mu_1 P) \le W_{C,b}(\mu_0, \mu_1).
  \end{equation*}
  In particular, taking $b=0$ we recover the Kuwada-type duality
  \begin{equation*}
    W_2(\mu_0 P, \mu_1 P)^2 \le C W_2(\mu_0, \mu_1)^2.
  \end{equation*}
\end{corollary}

The case $C=1$ of Corollary \ref{gradient-wasserstein-contraction} is
\cite[Theorem 8.24]{LMS2018}, and when additionally $b=0$ it reduces to
\cite[Proposition 3.7]{kuwada-duality}.

\begin{proof}
  Noting that $(Pf)^2 \le P(f^2)$ by Jensen's inequality, the gradient
  estimate $|\nabla Pf|^2 \le C P|\nabla f|^2$ implies that
  \eqref{poincare-type} holds with $a=1, \gamma=C, \delta=b$.  
\end{proof}

\begin{remark}
Note that, conversely, the estimate
 \begin{equation*}
    W_2(\mu_0 P, \mu_1 P)^2 \le C W_2(\mu_0, \mu_1)^2
  \end{equation*}
  implies the gradient estimate $|\nabla Pf|^2 \le C P|\nabla
  f|^2$; see \cite{kuwada-duality}.
\end{remark}

\begin{theorem}\label{equi-func}
Let $C >0$. The following are equivalent:
\begin{enumerate}
\item The reverse Poincar\'e inequality
 \begin{equation}\label{RPI}
    |\nabla P f|^2 \le C(P(f^2) - (Pf)^2), \qquad f \in \lip_b(X).
    \tag{RPI}
  \end{equation}
  \item  The Hellinger--Kantorovich contraction
  \begin{equation}\label{HKC}
    \mathsf{He}_2(\mu_0 P, \mu_1 P)^2 \le \HK_{4/C}(\mu_0, \mu_1)^2 \le
    \frac{C}{4} W_2(\mu_0, \mu_1)^2,\qquad \mu_0, \mu_1 \in \mathcal{P}(X).
    \tag{HKC}
     \end{equation}
    \item The Harnack type inequality
 \begin{equation}\label{HPI}
 P f(x) \le P f (y)+\sqrt{C}  d(x,y)  \sqrt{P (f^2) (x)}, \qquad x,y \in X, f  \in B_b(X), f \ge 0.
 \tag{HPI}
\end{equation}

\end{enumerate}
\end{theorem}

We point out, for future use, that the reverse Poincar\'e inequality
\eqref{RPI} is equivalent to the apparently weaker form
\begin{equation}\label{RPI-weak}
  |\nabla Pf|^2 \le C P(f^2), \qquad f \in \lip_b(X)
\end{equation}
Indeed, to see that \eqref{RPI-weak} self-improves to \eqref{RPI},
suppose $f \in \lip_b(X)$, fix an arbitrary $x \in X$, and let $g(y) = f(y) - Pf(x)$.  Then apply
\eqref{RPI-weak} to $g$ and evaluate at $x$.

Before we give the proof of the theorem, we state a lemma interesting in itself.

\begin{lemma}\label{increment P}
For any $f \in B_b(X)$, and $x,y \in X$,
\begin{equation*}
 | Pf (x) -Pf (y) |^2 \le 2 \mathsf{He}_2 (\delta_x P, \delta_y P)^2
 \left( P(f^2)(x) +P(f^2)(y)\right).
\end{equation*}
\end{lemma}

\begin{proof}
Let $m$ be a Borel measure such that both $\delta_x P$ and $\delta_y
P$ are absolutely continuous with respect to $m$. We denote
\[
P_m(x,\cdot)=\frac{\d \delta_xP}{\d m}, \quad P_m(y,\cdot)=\frac{\d
  \delta_yP}{\d m}.
\]
We have
\begin{align*}
 & |Pf (x) -Pf(y)|\\ =& \left| \int P_m(x,z) f(z) \,\d m(z) -\int
  P_m(y,z) f(z) \,\d m(z) \right| \\ =& \left| \int \sqrt{P_m(x,z)}
  \sqrt{P_m(x,z)} f(z) \,\d m(z) -\int \sqrt{P_m(y,z)} \sqrt{P_m(y,z)}
  f(z) \,\d m(z) \right| \\ \le & \left| \int \sqrt{P_m(x,z)}
  \sqrt{P_m(x,z)} f(z) \,\d m(z) -\int \sqrt{P_m(x,z)} \sqrt{P_m(y,z)}
  f(z) \,\d m(z) \right| \\ & + \left| \int \sqrt{P_m(x,z)}
  \sqrt{P_m(y,z)} f(z) \,\d m(z) -\int \sqrt{P_m(y,z)} \sqrt{P_m(y,z)}
  f(z) \,\d m(z) \right| \\ \le & \int \left| \sqrt{P_m(x,z)} -
  \sqrt{P_m(y,z)} \right|\sqrt{P_m(x,z)} f(z) \,\d m(z) \\ &+ \int
  \left| \sqrt{P_m(x,z)} - \sqrt{P_m(y,z)} \right|\sqrt{P_m(y,z)} f(z)
  \,\d m(z)
\end{align*}
Therefore, by the Cauchy--Schwarz inequality,
\begin{align*}
  |Pf (x) -Pf(y)|^2 &\le   \mathsf{He}_2 (\delta_x P, \delta_y P)^2 \left(
  \sqrt{P(f^2)(x)} +\sqrt{P(f^2)(y)}\right)^2 \\
  &\le 2  \mathsf{He}_2 (\delta_x P, \delta_y P)^2 \left( P(f^2)(x) +P(f^2)(y)\right).
 \end{align*}
\end{proof}

We are now ready for the proof of Theorem \ref{equi-func}.

\begin{proof}[Proof of Theorem \ref{equi-func}]
\eqref{RPI} $\implies$ \eqref{HKC}: This follows from Theorem
\ref{poincare-func-ineq} with $\gamma = 0$ and $b=\delta=C$.  We note
again that this direction is the essence of \cite[Theorem 5.4]{luise-savare}.

\eqref{HKC} $\implies$ \eqref{RPI}: Assume that
\[
 \mathsf{He}_2(\mu_0 P, \mu_1 P)^2 \le \frac{C}{4} W_2 ( \mu_0 , \mu_1 )^2.
 \]
 Then, for every $x,y \in X$,
 \[
 \mathsf{He}_2(\delta_x P, \delta_y P)^2 \le \frac{C}{4} d(x,y)^2.
 \]
 Therefore, from Lemma \ref{increment P} one deduces
\begin{align}\label{eq:poil}
 | Pf (x) -Pf (y) |^2 \le \frac{C}{2} d(x,y)^2 ( P(f^2)(x) +P(f^2)(y)).
\end{align}
Similarly, one has
\begin{align*}
| P(f^2) (x) -P(f^2) (y) |^2 &  \le \frac{C}{2} d(x,y)^2 ( P(f^4)(x) +P(f^4)(y)) \\
 &\le C d(x,y)^2 \| f \|_\infty^4,
\end{align*}
which implies that $P(f^2)$ is a continuous function.
Since
\[
|\nabla P f|(x) =\lim_{r \to 0} \sup_{0<d(x,y) \le r} \frac{ | Pf(x)-Pf(y)|}{d(x,y)},
\]
we may divide both sides of \eqref{eq:poil} by $d(x,y)^2$ and let $y
\to x$ to obtain
\begin{equation*}
  |\nabla Pf|(x) \le C P(f^2)(x)
\end{equation*}
which, as noted above, self-improves to \eqref{RPI}.

\eqref{RPI} $\implies$ \eqref{HPI} and \eqref{HPI} $\implies$
\eqref{RPI}: The proof follows from Proposition 1.3 in
\cite{MR3174217} so we omit it for conciseness.
\end{proof}

\section{Applications to convergence to equilibrium}\label{equilibrium-sec}

In this section, we focus on the applications of the transportation
type inequalities proven in Theorem \ref{poincare-func-ineq} as a
powerful tool to prove convergence to equilibrium for Markov
semigroups.  We will mostly focus on the applications of the
transportation inequality

\[
\mathsf{He}_2 ( \mu_0 P, \mu_1 P)^2 \le \frac{C}{4} W_2 ( \mu_0 , \mu_1 )^2,
\]
which, according to Theorem \ref{equi-func}, comes from the reverse
Poincar\'e inequality
\[
| \nabla P f |^2 \le C (P(f^2) - (Pf)^2).
\]
  The original Kuwada duality proved in Corollary
  \ref{gradient-wasserstein-contraction} relating the transportation
  inequality
\[
W_2^2 ( \mu_0 P, \mu_1 P)^2 \le C W_2^2 ( \mu_0 , \mu_1 )^2
\]
to the gradient bound
\[
| \nabla P f |^2 \le C P(| \nabla f|^2)
\]
was already illustrated as a tool to prove convergence to equilibrium
in \cite{MR3677826}, so we will spend less time on it.  Also, our
examples will be finite dimensional, though applications could be
given in an infinite dimensional framework as in Section
\ref{quasi-sec}. In particular, applications to stochastic partial
differential equations might be the object of a future work.

\subsection{Diffusions with $\Gamma_2 \ge 0$}

In this section, as an illustration of our general results, we first show how to recover the results of \cite{luise-savare}.
Let $\Delta$ be a locally subelliptic diffusion operator (see Section
1.2 in \cite{baudoin2018geometric} for a definition of local
subellipticity) on a smooth manifold $M$. For smooth functions $f,g: M
\rightarrow \mathbb{R}$, we can define the \textit{carr\'e du champ}
operator as the symmetric first-order bilinear differential form given
by:
\begin{align}\label{carre du champ}
\Gamma (f,g) :=\frac{1}{2} \left( \Delta(fg)-f\Delta g-g\Delta f \right).
\end{align}
We write $\Gamma(f)$ for $\Gamma(f,f)$.  (When $\Delta$ is the
Laplacian on $\mathbb{R}^n$ or on a Riemannian manifold, we have
$\Gamma(f) = |\nabla f|^2$.)  We assume that $\Delta $ is symmetric
with respect to some smooth measure $\mu$ (not necessarily finite),
which means that for every pair of smooth and compactly supported
functions $f,g \in C_0^\infty(M)$,
\[
\int_{M} g \Delta f\, \d\mu= \int_{M} f \Delta g \,\d\mu.
\]

There is an intrinsic distance associated to the operator $\Delta$
that we now describe.  An absolutely continuous curve $\gamma: [0,T]
\rightarrow M$ is said to be subunit for the operator $L$ if for every
smooth function $f : M \to \mathbb{R}$ we have $ \left| \frac{\d}{\d
  t} f ( \gamma(t) ) \right| \le \sqrt{ (\Gamma f) (\gamma(t)) }$.  We
then define the subunit length of $\gamma$ as $\ell_s(\gamma) = T$.
Given $x, y\in M$, we indicate then with
\[
S(x,y) :=\{\gamma:[0,T]\to M\mid \gamma\ \text{is subunit
  for}\ \Gamma,\, \gamma(0) = x,\, \gamma(T) = y\}
\]
and assume that  $S(x,y) \not= \emptyset$ for every $x, y\in M$. For
instance, if $L$ is an elliptic operator or if $L$ is a sum of squares
operator that satisfies H\"ormander's condition, then this assumption
is satisfied. Under this assumption,
\begin{equation}\label{ds}
d(x,y) := \inf\{\ell_s(\gamma)\mid \gamma\in S(x,y)\}
\end{equation}
defines a distance on $M$ and $(M,d)$ is by construction a length
space. The carr\'e du champ operator yields a strong upper gradient
structure on $(M,d)$ and from Theorem 1.12 in
\cite{baudoin2018geometric} one has
\[
d(x,y)=\sup \left\{ |f(x) -f(y) | , f \in  C^\infty(M) , \| \Gamma(f) \|_\infty \le 1 \right\},\ \ \  \ x,y \in M.
\]
We assume that the metric space $(M,d)$ is complete. In that case,
from Propositions 1.20 and 1.21 in \cite{baudoin2018geometric}, the
operator $\Delta$ is essentially self-adjoint on $C_0^\infty(M)$. The
semigroup in $L^2(M,\mu)$ generated by $\Delta$ will be denoted by
$(P_t)_{t \ge 0}$. The Bakry $\Gamma_2$ operator is defined as

\[
\Gamma_2 (f,g) =\frac{1}{2} \left( \Delta(\Gamma(f,g))-\Gamma(f,\Delta
g)-\Gamma (g,\Delta f) \right), \quad f,g \in C^\infty(M).
\]

\begin{theorem}\label{BE Poinc}
Assume that for every $f \in C^\infty(M)$, $\Gamma_2(f,f) \ge
0$. Then, for every $\nu_1,\nu_2 \in \mathcal{P}_2(M)$ and $t>0$,
\[
\mathsf{He}_2( \nu_1 P_t , \nu_2 P_t)^2 \le \frac{1}{8t} W_2 ( \nu_1, \nu_2)^2.
\]
Therefore, if the invariant measure $\mu$ is a probability measure which belongs to $\mathcal{P}_2(M)$, then  for every $x \in M$ and $t>0$,
\[
\mathsf{He}_2( \delta_{x} P_t , \mu)^2 \le \frac{1}{8t} W_2 ( \delta_{x}, \mu)^2
\]
and when $t \to +\infty$, $\delta_{x} P_t$ converges to $\mu$ in total variation for every $x \in M$.
\end{theorem}

\begin{proof}
It follows from Bakry-\'Emery calculus (see for instance
\cite[Proposition 3.3 (5)]{bakry-tata}) that since $\Gamma_2 \ge 0$
one has the following gradient bound that holds for bounded and
Lipschitz functions $f$,
\[
 \Gamma(P_t f) \le \frac{1}{2t} (P_t(f^2)-(P_tf)^2), \quad t > 0
\]
which yields the conclusion thanks to Theorem \ref{equi-func}.
\end{proof}

\begin{example}
An example where the theorem applies is the case where $\Delta$ is the
Laplace--Beltrami operator on a complete Riemannian manifold of
non-negative Ricci curvature. In that
case, the invariant measure $\mu$ is the Riemannian volume measure, and
the assumption $\Gamma_2 \ge 0$ is equivalent to the condition that the
Ricci curvature of $M$ is non-negative.
\end{example}

\begin{remark} More generally, if $\Delta$ is taken to be the operator generated by
  the Cheeger energy as in \cite{luise-savare}, so that $\Gamma =
  \mathsf{Ch}$, then the hypothesis of Theorem \ref{BE Poinc}
  essentially asks for
  $X$ to be an $RCD(0, \infty)$ space, and the conclusion is included
  in \cite[Theorem 5.2]{luise-savare}.  Indeed, the $\Gamma_2 \ge K$
  condition was already the key idea of the results of
  \cite{luise-savare}.  Our purpose in stating Theorem \ref{BE Poinc}
  is 
  to draw attention
  to the consequence that $\delta_x P_t$ converges in total variation
  to its equilibrium measure, at a rate no slower than $1/\sqrt{t}$.
\end{remark}

\begin{remark}
If $\Gamma_2 \ge a $, then, Bakry--\'Emery calculus also  yields the gradient bound
\[
 \Gamma(P_t f) \le e^{-2at} P_t (\Gamma(f)).
\]
which therefore implies from Theorem \ref{poincare-func-ineq} the
following contraction property in the $W_2$ distance:
\[
W_2 ( \nu_1 P_t, \nu_2 P_t)^2\le e^{-2at} W_2 ( \nu_1,\nu_2)^2,
\]
This appears in  \cite{kuwada-duality} and \cite{St-V}.
\end{remark}

\subsection{Subelliptic operators}

The assumption $\Gamma_2 \ge 0$ requires some form of ellipticity of
$\Delta$. In order to generalize the previous theorem to truly
subelliptic operators, one can make use of the generalized
$\Gamma$-calculus developed in \cite{BaudoinGarofalo, BB}. In addition
to the carr\'e du champ form $\Gamma$ defined in \eqref{carre du
  champ}, we assume that $M$ is endowed with another smooth symmetric
bilinear differential form, indicated with $\Gamma^Z$, satisfying for
$f,g \in C^\infty(M)$
\[
\Gamma^Z(fg,h) = f\Gamma^Z(g,h) + g \Gamma^Z(f,h),
\] 
and $\Gamma^Z(f) = \Gamma^Z(f,f) \ge 0$.  Let us assume that:

\begin{itemize}
\item[(H.1)] There exists an increasing
sequence $h_k\in C^\infty_0(M)$   such that $h_k\nearrow 1$ on
$M$, and \[
\|\Gamma (h_k)\|_{\infty} +\|\Gamma^Z (h_k)\|_{\infty}  \to 0,\ \ \text{as} \ k\to \infty.
\]
\item[(H.2)]  
For any $f \in C^\infty(M)$ one has
\[
\Gamma(f, \Gamma^Z(f))=\Gamma^Z( f, \Gamma(f)).
\] 
 \end{itemize}
Let us then consider
\begin{equation}\label{gamma2Z}
\Gamma^Z_{2}(f,g) = \frac{1}{2}\big[\Delta \Gamma^Z (f,g) - \Gamma^Z(f,
\Delta g)-\Gamma^Z (g,\Delta f)\big].
\end{equation}
As for $\Gamma$ and $\Gamma^Z$, we will freely use the notations  $\Gamma_2(f) = \Gamma_2(f,f)$, $\Gamma_2^Z(f) = \Gamma^Z_2(f,f)$.

\begin{theorem}
Let $\rho_1 \ge 0, \rho_2 >0$ and $\kappa >0$. Assume that for every $f \in C^\infty(M)$ and $\nu>0$
\begin{align}\label{generalizedCD}
\Gamma_2(f)+\nu \Gamma_2^Z(f) \ge \left(
\rho_1-\frac{\kappa}{\nu}\right) \Gamma(f) +\rho_2 \Gamma^Z(f).
\end{align}
 Then,  for every $\nu_1,\nu_2 \in \mathcal{P}_2(M)$ and $t>0$,
\[
\mathsf{He}_2( \nu_1 P_t , \nu_2 P_t)^2 \le \frac{1}{8t}\left(
1+\frac{2\kappa}{\rho_2} \right) W_2 ( \nu_1, \nu_2)^2.
\]
Therefore, if the invariant measure $\mu$ is a probability measure which belongs to
$\mathcal{P}_2(M)$, then for every $x \in M$ and $t>0$,
\[
\mathsf{He}_2( \delta_{x} P_t , \mu)^2 \le \frac{1}{8t} \left(
1+\frac{2\kappa}{\rho_2} \right) W_2 ( \delta_{x}, \mu)^2
\]
and when $t \to +\infty$, $\delta_{x} P_t$ converges to $\mu$ in total
variation for every $x \in M$.
\end{theorem}

\begin{proof}
It follows from Proposition 3.2 in \cite{BB} that
 \[
 \Gamma(P_t f) \le \frac{1}{2t} \left( 1+\frac{2\kappa}{\rho_2}
 \right) (P_t(f^2) - (P_t f)^2)
 \]
and thus the conclusion follows from Theorem \ref{equi-func}.
\end{proof}

\begin{example}
An example where this theorem applies is the case where $\Delta$ is
the sub-Laplacian  operator on a compact H-type sub-Riemannian
manifold, see \cite{BGMR}. In that case, the invariant measure $\mu$
is again the Riemannian volume measure and the assumption
\eqref{generalizedCD} is equivalent to the fact that the horizontal
Ricci curvature of $M$ is non-negative. This applies for instance to
the sub-Laplacian on the special unitary group $\mathrm{SU}(2)$, as
well as to compact quotients of the Heisenberg group $\mathbb{H}^3$.
\end{example}

\subsection{Non symmetric Ornstein--Uhlenbeck semigroups on Carnot groups}

In this section, we show that the method also applies to hypoelliptic
and non-symmetric diffusion operators. In particular we prove a
quantitative rate of convergence for the non-symmetric
Ornstein--Uhlenbeck semigroup on a Carnot group.

A Carnot group of step (or depth) $N$ is a simply connected Lie
group $\mathbb{G}$ whose Lie algebra can be written
\[
\mathfrak{g}=\mathcal{V}_{1}\oplus...\oplus \mathcal{V}_{N},
\]
where
\[
\lbrack \mathcal{V}_{i},\mathcal{V}_{j}]=\mathcal{V}_{i+j}
\]
and
\[
\mathcal{V}_{s}=0,\text{ for }s>N.
\]
From the above properties, it is of course seen that Carnot groups are nilpotent.
The number
\[
\mathfrak D=\sum_{i=1}^N i \dim \mathcal{V}_{i}
\]
is called the homogeneous dimension of $\mathbb{G}$.  On
$\mathfrak{g}$ we can consider the family of linear operators which
act by scalar multiplication $t^{i}$ on $\mathcal{V}_{i} $. These
operators are Lie algebra automorphisms, due to the grading, and
induce Lie group automorphisms $\Delta_t :\mathbb{G} \rightarrow
\mathbb{G}$ which are called the canonical dilations of
$\mathbb{G}$. It is easily seen that there exists on $\mathbb{G}$ a
complete and smooth vector field $D$ such that
\[
\Delta_t =e^{(\ln t) D}.
\]
This vector field $D$ is called the dilation vector field on $\mathbb{G}$. If $X$ is a left (or right) invariant smooth horizontal vector field on $\mathbb{G}$, we have for every $f \in C^\infty(\mathbb{G})$, and $t \ge 0$,
\[
X(f \circ \Delta_t)=t Xf \circ \Delta_t.
\]

\

Let us now pick a basis $V_1,...,V_d$ of the vector
space $\mathcal{V}_1$. The vectors $V_i$ can be seen as left
invariant vector fields on $\mathbb{G}$. In the sequel, these vector fields shall still be denoted by $V_1,...,V_d$. The left invariant sub-Laplacian on $\mathbb{G}$ is the operator:
\[
\sum_{i=1}^d V_i^2.
\]
It is essentially self-adjoint on the space of smooth and compactly
supported functions with respect to the Haar measure $\mu$ of
$\mathbb{G}$. The heat semigroup $(P_t)_{t\ge 0}$ on $\mathbb{G}$
generated by the sub-Laplacian, defined through the spectral theorem,
is then a Markov semigroup.

There are two different operators on $\mathbb{G}$ which are both
commonly referred to as Ornstein--Uhlenbeck operators; see
\cite{lust-piquard-ou} for a thorough comparison of the two types and
their properties.  We are interested here in the \emph{non-symmetric}
Ornstein Uhlenbeck operator defined by
\[
L=\sum_{i=1}^d V_i^2 -\alpha D
\]
where $\alpha >0$. This operator generates a Markov semigroup
$(Q_t)_{t \ge 0}$ which is given by the Mehler formula
\[
Q_t f= P_{\frac{1-e^{-\alpha t}}{\alpha}}  ( f \circ \Delta_{e^{-\alpha t}} ), \quad t \ge 0.
\]
It is clear that the probability measure $\delta_{\mathbf{e}}
P_{1/\alpha}$ is invariant by $Q_t$ where $\mathbf{e}$ denotes the
identity element in $\mathbb G$. Note that $\delta_{\mathbf{e}}
P_{1/\alpha}$ is the heat kernel measure started from $\mathbf{e}$ in
$\mathbb{G}$.  From known heat kernel estimates in Carnot groups (see
\cite{varopoulos}), one easily sees that the invariant measure
$\delta_{\mathbf{e}} P_{1/\alpha} \in \mathcal{P}_2(\mathbb G )$. The
next theorem proves exponentially fast convergence to equilibrium for
$Q_t$ with a quantitative rate.

\begin{theorem}
For every $ x\in \mathbb{G}$ and $t > 0$,
\[
\mathsf{He}_2 ( \delta_x Q_t , \delta_{\mathbf{e}} P_{1/\alpha})^2 \le \frac{\mathfrak D \alpha e^{-2\alpha t}}{2(1-e^{-\alpha t})} W_2( \delta_x  , \delta_{\mathbf{e}} P_{1/\alpha})^2.
\]
\end{theorem}

\begin{proof}
We denote by $\nabla_{\mathcal H}$ the horizontal gradient on $\mathbb G$ given by
\[
\nabla_\mathcal{H} f=\sum_{i=1}^d (V_i f) V_i.
\]
The following reverse Poincar\'e inequality was proved in \cite{BB2}:
\[
| \nabla_\mathcal{H} P_t f |^2 \le \frac{\mathfrak D}{2t}
(P_t(f^2)-(P_t f)^2).
\]
Since $Q_t f= P_{\frac{1-e^{-\alpha t}}{\alpha}}  ( f \circ \Delta_{e^{-\alpha t}} )$, one has
\[
\nabla_\mathcal{H} Q_t f =e^{-\alpha t}  \nabla_\mathcal{H}  P_{\frac{1-e^{-\alpha t}}{\alpha}}  ( f \circ \Delta_{e^{-\alpha t}} ).
\]
Thus,
\[
| \nabla_\mathcal{H} Q_t f |^2 \le \frac{\mathfrak D \alpha
  e^{-2\alpha t}}{2(1-e^{-\alpha t})}P_{\frac{1-e^{-\alpha
      t}}{\alpha}} ( (f \circ \Delta_{e^{-\alpha t}})^2
)=\frac{\mathfrak D \alpha e^{-2\alpha t}}{2(1-e^{-\alpha t})}
Q_t(f^2)
\]
and the conclusion follows as before from Theorem \ref{equi-func}.
\end{proof}

\begin{remark}
The above proof and \cite{BB2} show that if $\mathbb{G}$ is an H-type
group, then the constant $\frac{\mathfrak D \alpha e^{-2\alpha
    t}}{2(1-e^{-\alpha t})}$ can be improved into $\frac{\mathfrak D
  \alpha e^{-2\alpha t}}{2d(1-e^{-\alpha t})}$.
\end{remark}

\subsection{Langevin type dynamics driven by L\'evy processes}

In this subsection, we work in the space $X = \mathbb{R}^n$ with its
usual Euclidean distance and gradient.

Let $(N_t)_{t \ge 0}$ be a L\'evy process in $\mathbb R^n$, i.e. a
c\`adl\`ag stochastic process with stationary and independent
increments. We assume that $N_0=0$ a.s. and that for every $T>0$,
$\mathbb{E} \left( \sup_{t\in [0,T]} | N_t |^2 \right) <+\infty$. In
$\mathbb{R}^n$, we consider the following stochastic differential
equation with additive noise:
\begin{align}\label{sde-frac}
dX^x_t=-\nabla U(X^x_t) dt +dN_t, \quad X^x_0=x \in \mathbb R^n,
\end{align}
where $U:\mathbb{R}^n \to \mathbb{R}$ is a $C^2$ function. For
simplicity, we assume that $\nabla U$ is a Lipschitz function, so that
it is easily proved that \eqref{sde-frac} has a unique solution for
any $x \in \mathbb{R}^n$ which moreover satisfies for every $T>0$,
$\mathbb{E} \left( \sup_{t\in [0,T]} | X_t^x |^2 \right)
<+\infty$. For $t \ge 0$, we denote by $P_t$ the Markov kernel defined
by
\[
P_t f(x) =\mathbb{E}( f(X_t^x) ),
\]
so that $P_t (x,A)=\mathbb{P} (X_t^x \in A)$. It is a contraction
semigroup in $L^\infty(\mathbb{R}^n)$, and from the square
integrability we have that for every $\mu \in \mathcal{P}_2(\mathbb
R^n)$ and $t \ge 0$, $\mu P_t \in \mathcal{P}_2(\mathbb R^n)$.

\subsubsection{Convergence to equilibrium in the Kantorovich--Wasserstein distance}

Let $\nabla^2 U$ denote the Hessian of $U$.

\begin{theorem}
Assume that there exists $a>0$ such that $\nabla^2 U \ge a$ (uniformly
in the sense of quadratic forms). Then, there exists a unique
probability measure $\mu$ in the Wasserstein space
$\mathcal{P}_2(\mathbb R^n)$ such that for every $t \ge 0$, $\mu P_t =
\mu$. Moreover, for every $t \ge 0$, and $\nu \in
\mathcal{P}_2(\mathbb R^n)$ one has,
\[
W_2 ( \nu P_t, \mu )^2 \le e^{-2at}  W_2 ( \nu , \mu )^2.
\]

\end{theorem}

\begin{proof}
We proceed in several steps. 

\

\underline{\textbf{Step 1}}: \textit{Proving the Bakry--\'Emery type estimate}. 

\
 Let $J_t=\frac{\partial X_t^x}{\partial x}$ be the first variation process associated with equation \eqref{sde-frac}.
Since $P_tf(x)=\mathbb{E}( f(X_t^x))$, by the chain rule we have
\[
\nabla P_t f (x)=\mathbb{E}\left( J_t^*  \nabla f(X_t^x)\right).
\]
Therefore, by the Cauchy--Schwarz inequality,
\[
| \nabla P_t f (x) |^2 \le \mathbb{E}\left( | J_t^*|^2\right)
\mathbb{E}\left( | \nabla f(X_t^x) |^2\right).
\]
Since $\mathbb{E}\left( | \nabla f(X_t^x) |^2\right)=P_t (| \nabla f
|^2)(x)$, we are left to estimate $\mathbb{E}\left( |
J_t^*|^2\right)$. To this end, we observe that
\begin{align}\label{sde-J}
dJ_t=-\nabla^2 U(X^x_t) J_t dt,  \quad J_0=\mathbf{Id}_{\mathbb R^n}.
\end{align}
From the assumption $\nabla^2 U \ge a$ this yields
\[
| J_t^*|^2 \le e^{-2at }.
\]
One concludes $\mathbb{E}\left( | J_t^*|^2\right) \le e^{-2at }$ and therefore
\[
| \nabla P_t f (x) |^2 \le e^{-2at} P_t (| \nabla f |^2)(x).
\]
By Kuwada duality (Corollary \ref{gradient-wasserstein-contraction}),
this yields that for every $\nu_0,\nu_1 \in \mathcal{P}_2(\mathbb
R^n)$,
\begin{align}\label{contraction wasserstein}
W_2 ( \nu_0 P_t, \nu_1 P_t )^2 \le e^{-2at}  W_2 ( \nu_0 , \nu_1 )^2.
\end{align}

\

\underline{\textbf{Step 2}}: \textit{Proving the existence and uniqueness of the invariant measure}. 

\

Let $t >0$. Thanks to \eqref{contraction wasserstein}, the map $\nu
\to \nu P_t$ is a contraction from $\mathcal{P}_2(\mathbb R^n)$ into
itself. Since $\mathcal{P}_2(\mathbb R^n)$ is a complete metric space,
one deduces that it admits a unique fixed point; call it $\mu_t$. We
have then for every $t>0$ that $\mu_t P_t =\mu_t$. Composing with
$P_s$ yields $\mu_t P_t P_s =\mu_t P_s$. Since $P_t$ is a semigroup,
one has $P_t P_s=P_sP_t$. Therefore, $\mu_t P_s P_t=\mu_t P_s$ which
means that $\mu_t P_s $ is invariant for $P_t$. By uniqueness this
implies $\mu_t P_s=\mu_t$. Using now the uniqueness of the invariant
measure for $P_s$ yields $\mu_t=\mu_s$. As a conclusion, $\mu_t$ is
independent of $t$. We can call it $\mu$.

\

\underline{\textbf{Step 3}}: \textit{Concluding}. 

\

Using \eqref{contraction wasserstein} with $\nu_0=\nu$ and $\nu_1=\mu$ yields the expected result. 
\end{proof}

\subsubsection{Convergence to equilibrium in the Hellinger distance}

Our next application shows that in the diffusion case one can prove
convergence to equilibrium in the Langevin dynamics without assuming
coercivity of the Hessian of the potential (i.e. $\nabla^2 U \ge a
>0$). The price to pay is a convergence speed which is not exponential
but polynomial. We now assume that $(N_t)_{t \ge 0}$ is a Brownian
motion in $\mathbb R^n$. In that case, the invariant measure of
\eqref{sde-frac} is known explicitly, and is given up to a possible
normalization constant by $e^{-U(x)}dx$.

\begin{theorem}\label{BEL formula}
Assume that the normalized invariant measure $d\mu=\frac{1}{Z}
e^{-U(x)} dx$ is a probability measure with a finite second moment and
that $\nabla^2 U \ge 0$ ($U$ convex). Then, for every $x \in
\mathbb{R}^n$
\[
\mathsf{He}_2( \delta_{x} P_t , \mu)^2 \le \frac{1}{4t} W_2 ( \delta_{x}, \mu)^2.
\]
In particular, $X^x_t$ converges in total variation to $\mu$ when $t \to +\infty$.
\end{theorem}

\begin{proof}
From the Bismut--Elworthy--Li formula
\cite{bismut-large-deviations,elworthy-li-1994}, we have for every $v
\in \mathbb{R}^n$
\[
\langle \nabla P_t f (x) , v \rangle=\frac{1}{t} \mathbb{E} \left(
f(X^x_t) \int_0^t (J_s v) dN_s \right),
\]
where, as before, $J_t=\frac{\partial X_t^x}{\partial x}$ is the first
variation process associated with equation \eqref{sde-frac}. From
the Cauchy--Schwarz inequality, and the fact that $\nabla^2 U \ge 0$
implies $|J_t| \le 1$ a.s., one has
\begin{align*}
\mathbb{E} \left( f(X^x_t) \int_0^t (J_s v) dN_s \right)^2 & \le
\mathbb{E} \left( f(X^x_t)^2 \right) \mathbb{E} \left( \left(\int_0^t
(J_s v) dN_s \right)^2\right) \\ & \le \mathbb{E} \left( f(X^x_t)^2
\right) \mathbb{E} \left(\int_0^t | J_s v |^2 ds \right) \\ &\le t | v
|^2 \mathbb{E} \left( f(X^x_t)^2 \right) =t |v|^2 P_t (f^2)(x).
\end{align*}
One concludes that for every $v \in \mathbb{R}^n$,
\[
\langle \nabla P_t f (x) , v \rangle^2=\frac{1}{t}  |v|^2 P_t (f^2)(x).
\]
This yields
\[
| \nabla P_t f (x) |^2 \le \frac{1}{t}   P_t (f^2)(x)
\]
which is of the form \eqref{RPI-weak}.  As noted before, this
self-improves to \eqref{RPI} and thus we have the expected result by Theorem \ref{equi-func}.
\end{proof}

\begin{remark}
Theorem \ref{BEL formula} might also be proven using Theorem \ref{BE
  Poinc} above.  However, we wanted to illustrate the use of the
Bismut--Elworthy--Li formula as a tool to prove reverse Poincar\'e
inequalities.
\end{remark}

\section{R\'enyi-type divergences and functional inequalities}\label{renyi-sec}

\subsection{The dynamic dual formulation and basic properties}\label{T-def-sec}

The notions discussed in the previous section can be modified to give
a dynamic dual formulation of a family of ``entropic'' divergences on
$\mathcal{P}(X) \times \mathcal{P}(X)$, which we will denote by
$T_{a,b}$.  In the same way that the (squared) distances $W_{a,b}$ included the
Hellinger distance, the $T_{a,b}$ family will include the R\'enyi
divergence (in a different normalization); and where contractions of
$W_{a,b}$ were equivalent to reverse Poincar\'e inequalities, we will
show (Theorem \ref{entropic-equivalence}) that contractions of
$T_{a,b}$ are equivalent to reverse logarithmic Sobolev inequalities,
Wang-type Harnack inequalities, and integrated Harnack inequalities.

\begin{definition}\label{T-def}
  Let $a,b \ge 0$.  We denote by $\mathcal{E}_{a,b}$ the class of all
  \emph{positive}   functions $\varphi \in C^1([0,1], \lip_b(X))$,
  bounded and bounded away from $0$, satisfying the
  differential inequality
  \begin{equation}
    \partial_s \varphi_s + a \varphi_s |\nabla \ln \varphi_s|^2 + b
    \varphi_s \ln \varphi_s \le 0.
  \end{equation}
  Then for probability measures $\mu_1, \mu_2 \in \mathcal{P}(X)$ we set
  \begin{equation}\label{T-sup}
    T_{a,b}(\mu_0, \mu_1) = \sup_{\varphi \in \mathcal{E}_{a,b}}
      \left[\int_X \varphi_1\,\d\mu_1 - \int_X \varphi_0\,\d\mu_0\right]
  \end{equation}
\end{definition}

(We note here a slight abuse of terminology. The functions $T_{a,b}$
as defined above do not actually satisfy the definition of a
statistical divergence, but they have renormalized versions
$\widetilde{T}_{a,b}$, defined in \eqref{Ttilde-def} below, which are
divergences as shown in Proposition \ref{T-divergence}.  However, it
will be simpler in most cases to work with $T_{a,b}$ than with
$\widetilde{T}_{a,b}$, and we will continue to use the term ``divergence''
for either of the two when no confusion will result.)

\begin{remark}
  By writing $\varphi_s = e^{\psi_s}$, we could formulate Definition
  \ref{T-def} instead as
  \begin{equation*}
    T_{a,b}(\mu_0, \mu_1) = \sup\left\{ \int_X e^{\psi_1} \,\d\mu_1 -
    \int_X e^{\psi_0}\,\d\mu_0  : \partial_s \psi_x + a |\nabla
    \psi_s|^2 + b \psi_s \le 0\right\}.
  \end{equation*}
  In this notation the relevant Hamilon--Jacobi differential inequality more closely
  resembles \eqref{W-diff-ineq}, with a $0$th order term which is now
  linear instead of quadratic.  However, for comparison to functional
  inequalities as in Theorem \ref{entropic-func-ineq} below, the
  original formulation of Definition \ref{T-def} will be more convenient.
\end{remark}

\begin{notation}\label{pq-notation}
  For $b>0$, let $p = e^b$, let $q = p/(p-1)$ be the conjugate
  exponent of $p$, and set
  \begin{equation}
    C_b := \frac{1}{q} p^{-q/p} = \frac{1}{q} p^{1-q}.
  \end{equation}
\end{notation}

We will use this notation throughout the rest of the paper when
discussing $T_{a,b}$.  The reader should keep in mind $p,q, C_b$
depend implicitly on $b$.

The following elementary inequality will be used several times.
\begin{lemma}\label{elem-ineq}
Let $b > 0$ and define $p,q,C_b$ as in Notation \ref{pq-notation}. Suppose $z,w >
0$.  Then for all $x > 0$ we have
\begin{equation*}
  x^{1/p} z - xw \le C_b \frac{z^q}{w^{q-1}}
\end{equation*}
with equality when $\displaystyle x = \left(\frac{z}{pw}\right)^q$.
\end{lemma}

\begin{proof}
  We can suppose without loss of generality that $w=1$, for applying
  this case with $z$
  replaced by $z/w$ and multiplying through by $w$ yields the general
  case.

  Using Young's inequality for products $uv \le \frac{1}{p} u^p +
  \frac{1}{q} v^q$, we have
  \begin{equation*}
    x^{1/p} z - x = (px)^{1/p} (p^{-1/p} z) - x \le x + \frac{1}{q}
    p^{-q/p} z^q - x = C_b z^q.
  \end{equation*}
  Young's inequality becomes equality precisely when $u^p = v^q$,
  which in this case means $px = p^{-q/p} z^q$ or $x = p^{-q/p - 1}
  z^q$.  Since $-q/p - 1 =  -q$ this is
  the desired expression.

  Alternatively, one can let $y = -pw/z < 0$, $f(x) = -px^{1/p}$, and
  write
  \begin{equation*}
    \sup_{x > 0} x^{1/p} z - xw  = \frac{z}{p} \sup_{x > 0} (xy - f(x)) = \frac{z}{p} f^*(y)
  \end{equation*}
  where $f^*$ denotes the Legendre transformation or Fenchel conjugate of
  the convex function $f$.  It is known that $f^*(y) =
  -(-y)^{1-q}/(1-q)$ \cite[Table 3.1]{borwein-lewis} and this yields
    the desired statement.  We thank the anonymous referee for this
    observation.
\end{proof}

\begin{lemma}\label{T-generic-lower-bound}
  Suppose $a \ge 0$, $b > 0$.  Let $f : X \to \mathbb{R}$ be bounded
  and $1$-Lipschitz.  Then for every $k < \frac{1}{4a}$ we have
  \begin{equation}\label{T-generic-lower-eqn}
    T_{a,b}(\mu_0, \mu_1) \ge     C_b \frac{\left(\int_X \exp\left(\frac{kb}{e^b-4ak}
      f^2\right)\,\d\mu_1 \right)^q}
    {\left(\int_X \exp\left(\frac{kb}{\vphantom{e^b}1-4ak}
      f^2\right)\,\d\mu_0\right)^{q-1}}.
  \end{equation}
\end{lemma}

\begin{proof}
  We consider a function $\varphi$ of the form
  $\varphi_s = \exp(\alpha(s) f^2 + \beta(s))$.  In order to have $\varphi_s \in
  \mathcal{E}_{a,b}$ we require
  \begin{align*}
    0 &\ge \partial_s \varphi_s + a \varphi_s |\nabla \ln \varphi_s|^2
    + b \varphi_s \ln \varphi_s \\
    &= \varphi_s \cdot \left(\alpha'(s) f^2 + \beta'(s) + a
    \alpha(s)^2 |\nabla f^2|^2 + b \alpha(s) f^2 + b \beta(s)\right).
  \end{align*}
  Since $|\nabla f| \le
  1$, we have $|\nabla f^2|^2
  = (2f |\nabla f|)^2 \le 4f^2$, so it suffices to have
  \begin{align*}
    \alpha'(s) + 4 a \alpha(s)^2 + b \alpha(s) &= 0 \\
    \beta'(s) + b \beta(s) &= 0
  \end{align*}
which is satisfied by
\begin{align*}
  \alpha(s) &= \frac{kb}{e^{bs}-4ak} \\
  \beta(s) &= \beta_0 e^{-bs}
\end{align*}
for any $k < \frac{1}{4a}$ and any $\beta_0 \in \mathbb{R}$.  So with this
choice of $\varphi$, we have
\begin{align*}
  T_{a,b}(\mu_0, \mu_1) &\ge \int_X \varphi_1\,\d\mu_1 - \int_X
  \varphi_0\,\d\mu_0 \\
  &= e^{\beta_0 e^{-b}} \int_X \exp\left(\frac{kb}{e^b-4ak}
  f^2\right)\,\d\mu_1 - e^{\beta_0} \int_X \exp\left(\frac{kb}{1-4ak}
  f^2\right)\,\d\mu_0 \\
  &= C_b \frac{\left(\int_X \exp\left(\frac{kb}{e^b-4ak}
      f^2\right)\,\d\mu_1 \right)^q}
    {\left(\int_X \exp\left(\frac{kb}{\vphantom{e^b}1-4ak}
      f^2\right)\,\d\mu_0\right)^{q-1}}
\end{align*}
when we make an optimal choice of $\beta_0$ as described in Lemma \ref{elem-ineq},
with $x=e^{\beta_0}$ and noting that
$1/p=e^{-b}$.
\end{proof}

In the previous lemma, when $a=0$, the gradient terms vanish, and we
can instead consider a function $\varphi_s$ of the form $\varphi_s =
\exp(\alpha(s) f + \beta(s))$ where $f$ need only be bounded and
Lipschitz.  This yields the following improvement:

\begin{corollary}\label{T-generic-lower-a0}
  Suppose $b>0$, and let $f \in \lip_b(X)$.  Then
  \begin{equation}\label{T-lower-a0-eqn}
    T_{0,b}(\mu_0, \mu_1) \ge     C_b \frac{\left(\int_X e^{e^{-b} f}
      \,\d\mu_1 \right)^q}
    {\left(\int_X e^f \,\d\mu_0\right)^{q-1}}.
  \end{equation}
\end{corollary}

\begin{lemma}\label{gronwall-type}
  Let $b>0, r \ge 0$.  Suppose $y : [0,1] \to (0,\infty)$ is absolutely
  continuous and satisfies the differential inequality
  \begin{equation}\label{gronwall-diff-ineq}
    y' \le ry - b y \ln y \quad \text{a.e.}, \qquad y(0) = y_0 > 0.
  \end{equation}
  Then
  \begin{equation}\label{gronwall-conclusion}
    y(s) \le \exp\left(\frac{r}{b} \left(1-e^{-bs}\right)\right) y_0^{e^{-bs}}
  \end{equation}
  for all $0 \le s \le 1$, and in particular, following Notation \ref{pq-notation},
  \begin{equation}\label{gronwall-conclusion-t1}
    y(1) \le \exp\left(\frac{r}{qb}\right) y_0^{1/p}.
  \end{equation}
\end{lemma}

\begin{proof}
  Let $\zeta(s) = e^{bs} \ln y(s)$; then $\zeta(s)$ satisfies
  $\zeta'(s) \le r e^{bs}$ a.e.  Integrating from $0$ to $s$ yields
  \begin{equation}
    \zeta(s) \le \ln y_0 + \frac{r}{b} \left(e^{bs}-1\right)
  \end{equation}
  which rearranges to \eqref{gronwall-conclusion}.
\end{proof}

\begin{proposition}\label{T-divergence}
  For all $a \ge 0$ and $b > 0$, we have $T_{a,b}(\mu_0, \mu_1) \ge
  C_b$, with equality iff $\mu_0 = \mu_1$.  Thus
  \begin{equation}\label{Ttilde-def}
  \widetilde{T}_{a,b}(\mu_0, \mu_1) := \ln \frac{1}{C_b}
    T_{a,b}(\mu_0,\mu_1)
  \end{equation}
  is a statistical divergence on $\mathcal{P}(X)$; that is, 
  $\widetilde{T}_{a,b}(\mu_0, \mu_1) \ge 0$ with equality iff $\mu_0 = \mu_1$.
\end{proposition}

\begin{proof}
  The lower bound $T_{a,b}(\mu_0, \mu_1) \ge C_b$ follows from Lemma
  \ref{T-generic-lower-bound} with $f=0$.  

  To show equality holds when $\mu_0= \mu_1 = \mu$, let $\varphi \in
  \mathcal{E}_{a,b}$.  Note that in particular, $\varphi$ satisfies
  $\partial_s \varphi_s + b \varphi_s \ln \varphi_s \le 0$; that is,
  $\mathcal{E}_{a,b} \subseteq \mathcal{E}_{0,b}$.  So for each $x$,
  $y(s) = \varphi_s(x)$ satisfies \eqref{gronwall-diff-ineq} with
  $r=0$, and so by Lemma \ref{gronwall-type} and Lemma \ref{elem-ineq}
  we have
  \begin{equation*}
    \varphi_1(x) - \varphi_0(x) \le \varphi_0(x)^{1/p} - \varphi_0(x)
    \le C_b.
  \end{equation*}
  Thus $\int_X (\varphi_1 - \varphi_0)\,d\mu \le C_b$ and taking the
  supremum over $\varphi_s \in \mathcal{E}_{a,b}$ we have
  $T_{a,b}(\mu,\mu) \le C_b$.

  Conversely, suppose $\mu_0, \mu_1$ satisfy $T_{a,b}(\mu_0, \mu_1) =
  C_b$.  Let $f : X \to \mathbb{R}$ be bounded and 1-Lipschitz.  Lemma
  \ref{T-generic-lower-bound} then implies
  \begin{equation*}
    \left(\int_X \exp\left(\frac{kb}{e^b-4ak}
      f^2\right)\,\d\mu_1 \right)^q \le \left(\int_X \exp\left(\frac{kb}{\vphantom{e^b}1-4ak}
      f^2\right)\,\d\mu_0\right)^{q-1}
  \end{equation*}
  for every $k < \frac{1}{4a}$.  When $k=0$, both sides equal $1$, so we
  differentiate the inequality at $k=0$ to obtain
  \begin{equation*}
    q b e^{-b} \int_X f^2\,\d\mu_1 \le (q-1) b \int_X f^2\,\d\mu_0
  \end{equation*}
  which rearranges to
  \begin{equation*}
    \int_X f^2\,\d\mu_1 \le \int_X f^2\,\d\mu_0
  \end{equation*}
  since $q/(q-1) = p = e^b$.  The rest is a density argument.
  Replacing $f$ by $f+c$ for an arbitrary constant $c \in \mathbb{R}$ and expanding, we
  get
  \begin{equation*}
    \int_X f^2\,\d\mu_1 + 2c \int_X f\,\d\mu_1 + c^2 \le \int_X
    f^2\,\d\mu_0 + 2c \int_X f\,\d\mu_0 + c^2
  \end{equation*}
  Letting $c \to \pm \infty$, we see this implies $\int_X f\,\,d\mu_1 =
  \int_X f\,\d\mu_0$ for all bounded 1-Lipschitz $f$, and by scaling,
  the same holds for all bounded Lipschitz $f$.  This implies $\mu_0 = \mu_1$.
\end{proof}

We now show that when $a=0$, $\widetilde{T}_{0,b}$ recovers the R\'enyi
divergence, whose definition we recall:

\begin{definition}
  For $\mu_0, \mu_1 \in \mathcal{P}(X)$ and $r > 1$, the R\'enyi
  divergence of order $r$ is given by
  \begin{equation*}
    D_r(\mu_1 \,\|\,  \mu_0) := \frac{1}{r-1} \ln \int_X
  \left(\frac{\d \mu_1}{\d\mu_0}\right)^r\,\d\mu_0 = \frac{1}{r-1} \ln \int_X
  \left(\frac{\d \mu_1}{\d\mu_0}\right)^{r-1}\,\d\mu_1
  \end{equation*}
  if $\mu_1$ is absolutely continuous with respect to $\mu_0$, and $D_r(\mu_1
  \,\|\,  \mu_0) = \infty$ otherwise.
\end{definition}

\begin{lemma}\label{T-abs-cont}
  Let $b > 0$. If $T_{0,b}(\mu_0, \mu_1) < \infty$ then $\mu_1$ is
  absolutely continuous with respect to $\mu_0$.
\end{lemma}

\begin{proof}
  Let $A \subset X$ be a Borel set for which $\mu_0(A) = 0$.  We can
  then find a sequence of bounded nonpositive Lipschitz functions
  $f_n$ such that $f_n(x) \to 0$ for a.e. $x \in A$ (with respect to
  $\mu_0 + \mu_1$), and $f_n(x) \to -\infty$ for a.e. $x \in A^c$.
  Applying Corollary \ref{T-generic-lower-a0} to $f_n$, we have
  \begin{equation*}
    \left(\int_X e^{e^{-b} f_n}\,\d\mu_1\right)^q \le
    \frac{T_{0,b}(\mu_0,\mu_1)}{C_b} \left(\int_X e^{f_n}\,\d\mu_0\right)^{q-1}.
  \end{equation*}
  Letting $n \to \infty$ and using dominated convergence, this becomes
  \begin{equation*}
    \mu_1(A)^q \le \frac{T_{0,b}(\mu_0,\mu_1)}{C_b} \mu_0(A)^{q-1} = 0.
  \end{equation*}
\end{proof}

\begin{proposition}\label{T-renyi}
  Let $b>0$.  For all $\mu_0, \mu_1 \in \mathcal{P}(X)$ we have
  \begin{equation}
    T_{0,b}(\mu_0, \mu_1) = \begin{cases} C_b \int_X
      \left(\frac{\d\mu_1}{\d \mu_0}\right)^q\,\d\mu_0, & \mu_1 \ll
      \mu_0 \\
\infty, & \mu_1 \not\ll \mu_0
\end{cases}
  \end{equation}
so that
  \begin{equation}
    \widetilde{T}_{0,b}(\mu_0, \mu_1) = (q-1) D_q(\mu_1 \,\|\,  \mu_0).
  \end{equation}
\end{proposition}

\begin{proof}
  The case $\mu_1 \not\ll \mu_0$ is the contrapositive of Lemma \ref{T-abs-cont},
  so suppose $\mu_1 \ll \mu_0$ and let $\varrho =
  \frac{\d\mu_1}{\d\mu_0}$.  We show $T_{0,b} = C_b \int_X
  \varrho^q\,\d\mu_0$.

  To show $T_{0,b} \le C_b \int_X
  \varrho^q\,\d\mu_0$, let $\varphi \in \mathcal{E}_{0,b}$.  Taking
  $y(s) = \varphi_s(x)$ and $r=0$ in Lemma \ref{gronwall-type}, we have $\varphi_1 \le
  \varphi_0^{1/p}$ pointwise.  Hence
  \begin{align*}
    \int_X \varphi_1\,\d\mu_1 - \int_X \varphi_0\,\d\mu_0 &= \int_X
    \left(\varphi_1 \varrho - \varphi_0\right)\,\d\mu_0 \\ &\le \int_X
    \left(\varphi_0^{1/p} \varrho - \varphi_0\right)\,\d\mu_0 \\
    &\le C_b \int_X \varrho^q\,\d\mu_0
  \end{align*}
  by Lemma \ref{elem-ineq}.  Taking the supremum over $\varphi \in
  \mathcal{E}_{0,b}$ yields the desired upper bound.

  For the lower bound, if $\varrho \in L^q(\mu_0)$, take a sequence
  $f_n \in \lip_b(X)$ such that $e^{f_n} \to \varrho^q$,
  $\mu_0$-almost everywhere (hence also $\mu_1$-almost everywhere) and
  in $L^1(\mu_0)$.  Then Corollary \ref{T-generic-lower-a0} gives
  \begin{equation*}
    C_b \left(\int_X e^{e^{-b} f_n}\,\d\mu_1\right)^q \le
    T_{0,b}(\mu_0,\mu_1) \left(\int_X e^{f_n}\,\d\mu_0\right)^{q-1}.
  \end{equation*}
  Pass to the limit, applying Fatou's lemma on the left and $L^1$
  convergence on the right, to obtain
  \begin{equation}\label{varrho-lower}
    C_b \left(\int_X \varrho^{q e^{-b}} \,\d\mu_1\right)^q \le
    T_{0,b}(\mu_0,\mu_1) \left(\int_X \varrho^q\,\d\mu_0\right)^{q-1}.
  \end{equation}
  Now observe that $qe^{-b} = q/p = q-1$ and so $\int_X \varrho^{q
    e^{-b}} \,\d\mu_1 = \int_X
  \varrho^{q-1}\,\d\mu_1 = \int_X \varrho^{q}\,\d\mu_0$.  Hence
  \eqref{varrho-lower} rearranges to
  \begin{equation*}
    C_b \int_X \varrho^q \,\d\mu_0 \le T_{0,b}(\mu_0,\mu_1)
  \end{equation*}
  as desired.

  If $\varrho \notin L^q(\mu_0)$, then we need to show $T_{0,b}(\mu_0,
  \mu_1) = \infty$.  Let $m \ge 0$ and $A_m = \{\varrho \le m\}$.  Choose $f_n \in \lip_b(X)$ with $e^{f_n} \to
  \varrho^q 1_{A_m}$, $\mu_0$-a.e. and in $L^1(\mu_0)$.  Then
  proceeding as in the previous case, we obtain
  \begin{equation*}
    C_b \int_{A_m} \varrho^q \,\d\mu_0 \le T_{0,b}(\mu_0, \mu_1).
  \end{equation*}
  Letting $m \to \infty$ and applying the monotone convergence
  theorem, we conclude that $T_{0,b}(\mu_0, \mu_1) = +\infty$.
\end{proof}

Finally, we estimate the value of $T_{a,b}$ for point masses.

\begin{proposition}\label{T-point-mass-improved}
  Suppose $a,b > 0$, $x_0, x_1 \in X$.  Then
  \begin{equation}\label{T-point-mass-eqn-improved}
    C_b \exp\left(\frac{bq}{4a(p-1)} d(x_0,x_1)^2\right) \le
    T_{a,b}(\delta_{x_0}, \delta_{x_1}) \le C_b
    \exp\left(\frac{1}{4ab}d(x_0, x_1)^2\right)
  \end{equation}
  or in terms of $\widetilde{T}_{a,b}$,
  \begin{equation}\label{Ttilde-point-mass-eqn-improved}
    \frac{bq}{4a(p-1)} d(x_0,x_1)^2 \le
    \widetilde{T}_{a,b}(\delta_{x_0}, \delta_{x_1}) \le
    \frac{1}{4ab}d(x_0, x_1)^2.
  \end{equation}
\end{proposition}

\begin{proof}
For the upper bound, suppose  $\varphi \in \mathcal{E}_{a,b}$, and as in the proof of the
upper bound in 
Proposition \ref{W-Dirac-distance-prop}, let $\gamma : [0,1] \to X$ be
a constant speed geodesic joining $x_0$ to $x_1$.  Using the chain
rule, we have
\begin{align*}
  \frac{d}{ds} \varphi_s(\gamma_s) &\le \partial_s \varphi_s(\gamma_s) + |\nabla
  \varphi_s|(\gamma_s) d(x_0, x_1) \\
    &\le -\frac{a}{\varphi_s(\gamma_s)} |\nabla
    \varphi_s|(\gamma_s)^2 - b \varphi_s(\gamma_s) \ln
    \varphi_s(\gamma_s) + |\nabla
    \varphi_s|(\gamma_s) d(x_0, x_1) \\
    &\le \frac{d(x_0, x_1)^2}{4a} \varphi_s(\gamma_s) -
      b \varphi_s(\gamma_s) \ln \varphi_s(\gamma_s)
\end{align*}
by completing the square.  So $y(s) = \varphi_s(\gamma_s)$ satisfies
the differential inequality \eqref{gronwall-diff-ineq} with $r=d(x_0,
x_1)^2/4a$, and by Lemma \ref{gronwall-type} and Lemma \ref{elem-ineq}
we have
\begin{align*}
  \varphi_1(x_1) - \varphi_0(x_0) &\le \exp\left(\frac{1}{4abq}d(x_0,
    x_1)^2\right) \varphi_0(x_0)^{1/p} - \varphi_0(x_0) \\
  &\le C_b \exp\left(\frac{1}{4ab}d(x_0,
    x_1)^2\right).
\end{align*}

For the lower bound, apply Lemma \ref{T-generic-lower-bound} with
$\mu_i = \delta_{x_i}$ and $f(x) = d(x_0, x) \wedge d(x_0, x_1)$,
which is bounded and $1$-Lipschitz.  Since $f(x_0) = 0$, the $\d\mu_0$
integral in \eqref{T-generic-lower-eqn} equals $1$, and we obtain
\begin{equation*}
  T_{a,b}(\delta_{x_0}, \delta_{x_1}) \ge C_b \exp\left(\frac{qkb}{e^b
    - 4ak} d(x_0, x_1)^2\right)
\end{equation*}
for any $k < 1/4a$.  Letting $k \uparrow 1/4a$ and recalling that
$e^b=p$, we have the desired inequality.
\end{proof}

\begin{corollary}
  For $\mu_0, \mu_1 \in \mathcal{P}(X)$, we have
  \begin{equation}\label{T-primal}
    T_{a,b}(\mu_0, \mu_1) \le C_b \inf_{\pi} \int_{X \times X}
    \exp\left(\frac{1}{4ab} d(x,y)^2\right) \pi(\d x, \d y).
  \end{equation}
  where the infimum is taken over all couplings $\pi \in \mathcal{P}(X
  \times X)$ of $\mu_0, \mu_1$.  As a special case, we have
  \begin{equation}
    T_{a,b}(\delta_x, \mu) \le C_b \int_X \exp\left(\frac{1}{4ab}
    d(x,y)^2\right) \,\mu(\d y).
  \end{equation}
\end{corollary}

\begin{proof}
  Let $\varphi \in \mathcal{E}_{a,b}$ and let $\pi$ be a coupling of
  $\mu_0, \mu_1$.  Then we have
    \begin{align*}
    \int_X \varphi_1\,\d\mu_1 - \int_X \varphi_0 \,\d\mu_0 &= \int_{X
      \times X}
    (\varphi_1(y) - \varphi_0(x))\,\pi(\d x, \d y) \\
    &\le \int_{X \times X} T_{a,b}(\delta_x, \delta_y)\,\pi(\d x, \d y) \\
    &\le C_b \int_{X \times X} \exp\left(\frac{1}{4ab} d(x,y)^2\right)
    \pi(\d x, \d y)
    \end{align*}
    and \eqref{T-primal} follows by taking the supremum over
    $\varphi$ and the infimum over $\pi$.
\end{proof}

\subsection{Functional inequalities}\label{T-func-ineq-sec}

In the same way that the Hellinger--Kantorovich contraction property
was equivalent to a reverse Poincar\'e inequality, it turns out that a
similar contraction property for $T_{a,b}$ is
equivalent to a reverse logarithmic Sobolev inequality, as well as to
a Wang-type Harnack inequality.

For one direction of this equivalence, the key tool is the following
general statement, analogous to Theorem \ref{poincare-func-ineq}.

\begin{theorem}\label{entropic-func-ineq}
  Let $a,b,\gamma,\delta \ge 0$.  Suppose that for all $f \in
  \lip_b(X)$ with $f > 0$, we have
  \begin{equation}\label{log-ineq}
 a ( P f ) | \nabla \ln P f |^2 + b ( P f ) \ln Pf \le \gamma P ( f | \nabla \ln f |^2 )+ \delta P ( f \ln f ). 
  \end{equation}
  Then for all $\mu_0, \mu_1 \in \mathcal{P}(X)$ we have
  \begin{equation}
    T_{\gamma,\delta}(\mu_0 P, \mu_1 P) \le T_{a,b}(\mu_0, \mu_1).
  \end{equation}
\end{theorem}

\begin{proof}
  Suppose that $\varphi \in \mathcal{E}_{\gamma, \delta}$.  Then we
  have
  \begin{align*}
   \partial_s P \varphi_s + a P\varphi_s |\nabla \ln P\varphi_s|^2 +
    b P\varphi_s \ln P\varphi_s  &\le \partial_s P \varphi_s + \gamma
    P(\varphi_s |\nabla \ln \varphi_s|^2) + \delta P(\varphi_s \ln
    \varphi_s) \\
    &= P\left(\partial_s \varphi_s + \gamma \varphi_s |\nabla \ln
    \varphi_s|^2 + \delta \varphi_s \ln
    \varphi_s\right) \\
    &\le 0
  \end{align*}
  since $\varphi_s \in \mathcal{E}_{\gamma, \delta}$ and $P$ is
  positivity preserving.  Thus $P\varphi_s \in \mathcal{E}_{a,b}$, and
  so
  \begin{equation*}
    \int \varphi_1 \,\d (\mu_1 P) - \int \varphi_0\,\d (\mu_0 P) =
    \int P\varphi_1\,\d\mu_1 - \int P\varphi_0\,\d\mu_0 \le T_{a,b}(\mu_0,\mu_1). 
  \end{equation*}
  Taking the supremum over $\varphi \in \mathcal{E}_{\gamma, \delta}$
  we have $T_{\gamma,\delta}(\mu_0 P, \mu_1 P)\le T_{a,b}(\mu_0, \mu_1)$. 
 \end{proof}

\begin{theorem}\label{entropic-equivalence}
    Let $C > 0$.  The following are equivalent:
    \begin{enumerate}
      \item The reverse logarithmic Sobolev inequality
        \begin{equation}\label{rlsi}
          Pf |\nabla \ln Pf|^2 \le C\left(P(f \ln f) - (Pf) \ln Pf\right),
          \qquad f \in \lip_b(X),\,f > 0.
          \tag{rLSI} 
        \end{equation}
    \item The family of entropic transportation-cost inequalities
      \begin{equation}\label{eti}
        T_{0, \kappa C}(\mu_0 P, \mu_1 P) \le T_{\kappa, \kappa C}(\mu_0,
        \mu_1), \qquad \mu_0, \mu_1 \in \mathcal{P}(X), \quad \kappa >0.
        \tag{ETI}
      \end{equation}
    \item The Wang-type Harnack inequality
        \begin{equation}\label{whi}
          \begin{split}
            P f(x)^{p} \le \exp\left( \frac{p
            }{p-1} \frac{C d(x,y)^2}{4} \right) P(f^p)(y), \\
            p > 1, \quad f \in \lip_b(X),\, f > 0, \quad x,y \in X.
          \end{split} \tag{WHI}
        \end{equation}
    \end{enumerate}
  \end{theorem}

  \begin{proof}
      \eqref{rlsi} $\implies$ \eqref{eti}: Apply Theorem
      \ref{entropic-func-ineq} with $a = \kappa$,  $b = \delta = \kappa
      C$, $\gamma = 0$.    

    \eqref{eti} $\implies$ \eqref{whi}: Fix $x,y \in X$, and let $m$
    be some finite reference measure so that $\delta_x P, \delta_y P$
    are absolutely continuous with respect to $m$, with densities $p_x, p_y$
    respectively.  Let $\kappa$ be arbitrary and let $p = e^{\kappa
      C}$, $q = p/(p-1)$.  By Proposition \ref{T-renyi}
    and Proposition \ref{T-point-mass-improved} with $a=\kappa$,
    $b=\kappa C$, taking $\mu_0 = \delta_y, \mu_1 = \delta_x$ to match
    notation with other papers, we have that
    \eqref{eti} implies the integrated Harnack inequality
    \begin{equation}
      \int \left(\frac{p_x}{p_y}\right)^q p_y\,\d m  = \int \left(\frac{p_x}{p_y}\right)^{q-1} p_x\,\d m \le
      \exp\left(\frac{1}{\kappa^2 C^2} \frac{C d(x,y)^2}{4}\right).
    \end{equation}
    Now $q-1 = 1/(p-1)$ so this may be rewritten as
    \begin{equation}\label{weak-ihi}
      \left(\int \left(\frac{p_x}{p_y}\right)^{1/(p-1)} p_x\,\d m\right)^{p-1} \le
      \exp\left(\frac{p-1}{(\log p)^2} \frac{C d(x,y)^2}{4}\right),
      \qquad p>1.
    \end{equation}
    Since $\kappa$ was arbitrary, \eqref{weak-ihi} holds for all $p>1$.
    By an application of H\"older's inequality (see \cite[Lemma 2.11]{bgm}), 
    \eqref{weak-ihi} implies the Wang-type
    Harnack inequality
    \begin{equation}\label{weak-whi}
      Pf(x)^p \le  \exp\left(\frac{p-1}{(\log p)^2} \frac{C
        d(x,y)^2}{4}\right) P(f^p)(y), \qquad p > 1.
    \end{equation}
    To recover the more usual form of the Wang Harnack
    inequality \eqref{whi}, we would like to have \eqref{weak-whi} with
    $\frac{p}{p-1}$ in the exponent in place of $\frac{p-1}{(\log
      p)^2}$.  To this end, fix $\epsilon > 0$.  As $p \to 1$, we have $
    \frac{p}{p-1} \sim \frac{p-1}{(\log p)^2}$, so for all
    sufficiently small $p' > 1$ we have $ \frac{p'-1}{(\log p')^2} \le
    (1+\epsilon) \frac{p'}{p'-1}$ and thus
    \begin{equation}\label{whi-epsilon}
      Pf(x)^{p'} \le  \exp\left(\frac{p'}{p'-1} \frac{(1+\epsilon)C
        d(x,y)^2}{4}\right) P(f^{p'})(y).
    \end{equation}
    In particular this holds for $p' = p^{1/n}$ for sufficiently large
    $n$.  From \cite[Proposition 2.1]{wang-harnack-2010}, with
    $(1+\epsilon)C$ in place of $C$, it follows
    that \eqref{whi-epsilon} holds for $p$ in place of $p'$, and
    letting $\epsilon \to 0$ we obtain \eqref{whi}.
    
    \eqref{whi} $\implies$ \eqref{rlsi}: This is shown, in essence, in
    \cite[Theorem 2.1]{wang-dim-free-2006}.  The proof there is in the
    setting of a manifold with bounded curvature, and requires some
    minor changes to apply in this setting, so we give the details.

    Let $f$ be a positive bounded Lipschitz function which is bounded
    away from $0$.  Observe first that \eqref{whi} implies that $Pf$
    is continuous.  To see this, fix $x \in X$ and $p > 1$.  Letting
    $y \to x$ in \eqref{whi}, we see that $Pf(x)^p \le \liminf_{y \to
      x} P(f^p)(y)$. Now as $p \to 1$ we have $f^p \to f$ uniformly,
    and since $P$ is Markovian we also have $P (f^p) \to Pf$
    uniformly.  So we can pass to the limit to conclude $Pf(x) \le
    \liminf_{y \to x} Pf(y)$.  For the other direction, apply
    \eqref{whi} with $x$ and $y$ interchanged.  Let $y \to x$ to
    obtain $\left(\limsup_{y \to x} Pf(y)\right)^p \le P(f^p)(x)$, and
    then let $p \to 1$.
    
    Now, by definition of $|\nabla Pf|$ there exists a sequence $y_n
    \to x$, with $y_n \ne x$, such that $\frac{Pf(y_n) - Pf(x)}{d(y_n, x)} \to \pm
    |\nabla Pf|(x)$.  Suppose first that we can choose $y_n$ so
    that $\frac{Pf(y_n) - Pf(x)}{d(y_n, x)} \to - |\nabla Pf|(x)$.
    Set $r_n = d(y_n, x)$ for convenience, and let $\delta > 0$ be
    arbitrary.  Then \eqref{whi} with $y = y_n$ and $p = 1+r_n \delta$
    reads
    \begin{equation*}
      Pf(x)^{1+r_n \delta} \le \exp\left(\frac{C}{4 \delta} (r_n+r_n^2
      \delta)\right) P(f^{1+r_n \delta})(y_n).
    \end{equation*}
    Subtracting $Pf(x)$, dividing by $r_n$, and breaking up the right
    side, we have
    \begin{align*}
      \frac{Pf(x)^{1+r_n \delta}-Pf(x)}{r_n}  &\le \frac{\exp\left(\frac{C}{4 \delta} (r_n+r_n^2
        \delta)\right) - 1}{r_n} P(f^{1+r_n \delta})(y_n)  \\
      &\quad + P\left(\frac{f^{1+r_n \delta} - f}{r_n}\right)(y_n)  + \frac{P f(y_n) - Pf(x)}{r_n}.
    \end{align*}
    We now pass to the limit.  Since $f$ is continuous, bounded, and
    bounded away from $0$, we have $f^{1+r_n \delta}
    \to f$ and  $\frac{1}{r_n} (f^{1+r_n \delta} - f) \to \delta
    f \ln f$ uniformly, and so the same
    is true when $P$ is applied.  We obtain
    \begin{equation}\label{rlsi-delta}
      \delta Pf(x) \ln Pf(x) \le \frac{C}{4 \delta} Pf(x) + \delta P(f \ln
      f)(x) - |\nabla Pf|(x)
    \end{equation}
    and now optimizing over $\delta$ and rearranging yields \eqref{rlsi}.

    Otherwise, there exists a sequence $y_n \to x$ such that
    $\frac{Pf(y_n) - Pf(x)}{d(y_n, x)} \to + |\nabla Pf|(x)$.  We
    apply \eqref{whi} with $x$ and $y$ interchanged and proceed as
    before to obtain
    \begin{align*}
      \frac{Pf(y_n)^{1+r_n \delta}-Pf(y_n)}{r_n}  &\le \frac{\exp\left(\frac{C}{4 \delta} (r_n+r_n^2
        \delta)\right) - 1}{r_n} P(f^{1+r_n \delta})(x)  \\
      &\quad + P\left(\frac{f^{1+r_n \delta} - f}{r_n}\right)(x)  + \frac{P f(x) - Pf(y_n)}{r_n}.
    \end{align*}
    Passing to the limit again yields \eqref{rlsi-delta}.  On the left side, we use the fact that since $Pf$ is continuous, bounded, and bounded
    away from $0$, we have $\frac{1}{r_n}((Pf)^{1+r_n\delta} - Pf) \to
    \delta Pf \ln Pf$ uniformly.
  \end{proof}

  \begin{remark}\label{integrated-harnack}
    The Wang Harnack inequality \eqref{whi} is also known to be
    equivalent to the integrated Harnack inequality
    \begin{equation}\label{ihi}
      \int_X \left(\frac{p_x}{p_y}\right)^{1/(p-1)} p_x\,\d m \le
      \exp\left(\frac{p}{(p-1)^2} \frac{C d(x,y)^2}{4}\right), \qquad
      p > 1 \tag{IHI}
    \end{equation}
    where as above $p_x, p_y$ are the densities of $\delta_x P,
    \delta_y P$ with respect to some reference measure $m$; see
    \cite[Lemma 2.11]{bgm}.  Hence \eqref{ihi} is also equivalent to
    \eqref{eti} and \eqref{rlsi}.  In the proof of Theorem
    \ref{entropic-equivalence}, we obtained \eqref{weak-ihi} which is
    infinitesimally weaker than \eqref{ihi}; the self-improvement
    comes via the application of \cite[Proposition
      2.1]{wang-harnack-2010}, applying \eqref{whi} along a sequence
    of points between $x$ and $y$.
  \end{remark}
  
  A different application of Theorem \ref{entropic-func-ineq} relates
  a gradient bound for $P$ to another type of contraction inequality
  for $T_{a,b}$, analogous to Theorem \ref{gradient-wasserstein-contraction}.
  
\begin{proposition}\label{entropic-gradient}
  Suppose that for some $C$, the operator $P$ satisfies the $L^1 \ln L$-type
  gradient estimate
  \begin{equation}\label{L1lnL}
    Pf |\nabla \ln Pf|^2 \le C P \left( f |\nabla \ln f|^2\right),
    \qquad f \in \lip_b(X), f > 0.
  \end{equation}
  Then for every $\kappa, \epsilon > 0$ we have
  \begin{equation}\label{T-gradient}
    T_{\kappa C,\epsilon}(\mu_0 P, \mu_1 P) \le T_{\kappa,
      \epsilon}(\mu_0, \mu_1), \qquad \mu_0, \mu_1 \in \mathcal{P}(X).
  \end{equation}
  In particular, this holds if we have the stronger $L^1$-type
  gradient estimate
  \begin{equation}\label{L1-gradient}
    |\nabla Pf| \le C^{1/2} P|\nabla f|, \qquad f \in \lip_b(X).
  \end{equation}
\end{proposition}

\begin{proof}
  By Jensen's inequality we have $(Pf) \ln Pf \le P(f \ln f)$, and
  combining this with \eqref{L1lnL} we have that \eqref{log-ineq}
  holds with $a = \kappa$, $\gamma = \kappa C$, $b=\delta=\epsilon$.
  The conclusion then follows from Theorem \ref{entropic-func-ineq}.

  To see that \eqref{L1-gradient} implies \eqref{L1lnL}, using the
  former together with the bivariate
  Jensen inequality for the convex function $\psi(x,y) = x^2/y$, we
  obtain
  \begin{equation*}
    \frac{|\nabla Pf|^2}{Pf} \le C \frac{P(|\nabla f|)^2}{Pf} \le C
    P\left(\frac{|\nabla f|^2}{f}\right)
  \end{equation*}
  which is equivalent to \eqref{L1lnL} thanks to the chain rule (Lemma \ref{chain-rule}).
\end{proof}

\begin{remark}
  It might seem more natural to take $\epsilon = 0$ in
  \eqref{T-gradient}, but in fact that statement would have no
  content, as one can show that $T_{a,0}(\mu_0,
  \mu_1) = +\infty$ for all $\mu_0 \ne \mu_1$.
\end{remark}

\section{Applications to quasi-invariance}\label{quasi-sec}

The reverse logarithmic Sobolev inequality \eqref{rlsi} has been the
object of significant study in the literature, although not nearly as
much as the ``forward'' logarithmic Sobolev inequality.  One
particularly interesting area of application is in proving absolute
continuity of heat kernel measures; especially in the presence of
group structure, where it can be used to show \emph{quasi-invariance}
of a heat kernel measure under group translation.  Such results are
commonly obtained through the use of the Wang Harnack inequality
\eqref{whi}, which as noted in Section \ref{T-func-ineq-sec} is
equivalent to \eqref{rlsi}.  In this section, we consider some
examples and show how the entropic transportation-cost inequality
\eqref{eti} provides an alternate route to these conclusions.

Although in this paper we limit our attention to a few specific known
results, there are many other situations where similar questions about
absolute continuity could be considered, especially in stochastic
PDE, see \cite{MR3099948}.  The techniques developed in this paper may be useful in the
study of these problems, and we hope to address this in future work.

\subsection{Subelliptic heat kernels on finite-dimensional Lie groups}

Let $\mathbb{G}$ be a finite-dimensional connected real Lie group with
identity element $\mathbf{e}$, and
suppose that $\mathbb{G}$ is equipped with a left-invariant
sub-Riemannian geometry: a bracket-generating left-invariant
sub-bundle $\mathcal{H} \subset T\mathbb{G}$, and a sub-Riemannian
metric $g$ which is a left-invariant inner product on $\mathcal{H}$.
We denote by $\nabla$ the horizontal sub-gradient, and $|\nabla f| :=
\sqrt{g(\nabla f, \nabla f)}$.  Let $d$ be the Carnot--Carath\'eodory
distance on $\mathbb{G}$; by the Chow--Rashevskii theorem, the
bracket-generating condition implies that $d(x,y) < \infty$ for all
$x,y \in \mathbb{G}$.  Let $L$ be the left-invariant sub-Laplacian
induced by $g$, $P_t = e^{tL}$ the heat semigroup generated by $L$,
and $\mu_t = \delta_{\mathbf{e}} P_t$ the heat kernel measure.

Under these conditions, H\"ormander's theorem implies that $L$ is
subelliptic and hence $\mu_t$ is a smooth measure for all $t>0$.  Our
purpose here is to remark that at least part of this conclusion can be
recovered using our techniques instead, if one has a reverse log Sobolev
inequality.

Recall that in general, a Borel probability measure $\mu$ on a topological
group $\mathbb{G}$ is said to be \emph{quasi-invariant} under left
translation by an element $x \in \mathbb{G}$ if $\mu$ and its left translation $\mu^x(A)
= \mu(x^{-1} A)$ are mutually absolutely continuous.  If this holds
for every $x$ in some subgroup $H \subset \mathbb{G}$, we say $\mu$ is
quasi-invariant under left translation by $H$.

\begin{proposition}\label{lie-quasi}
  Suppose, under the above assumptions, that $P_t$ satisfies the reverse logarithmic Sobolev
  inequality
  \begin{equation}\label{rlsi-lie-group}
    P_t f |\nabla P_t f|^2 \le C(t) (P_t(f \ln f) - (P_t f) \ln P_t f). 
  \end{equation}
  Then for all $t>0$, $\mu_t$ is quasi-invariant under translation by
  every $x \in \mathbb{G}$.  As a consequence, $\mu_t$ is absolutely
  continuous with respect to left Haar measure and has full support.
\end{proposition}

\begin{proof}
  By Theorem \ref{entropic-equivalence}, \eqref{rlsi-lie-group}
  implies the entropic transportation-cost inequality
  \begin{equation*}
    T_{0, \kappa C(t)} (\mu_0 P_t, \mu_1 P_t) \le T_{\kappa, \kappa
      C(t)}(\mu_0, \mu_1), \qquad \kappa > 0. 
  \end{equation*}
  Taking $\mu_0 = \delta_{\mathbf{e}}$,
  $\mu_1 = \delta_x$ and applying Lemma \ref{T-point-mass-improved} to
  bound $T_{\kappa, \kappa C(t)}(\delta_{\mathbf{e}}, \delta_x)$, 
  we find that $T_{0, b}(\mu_t, \mu_t^x) < \infty$, and so
  Lemma \ref{T-abs-cont} implies that $\mu_t^x \ll \mu_t$; the
  opposite relation $\mu_t \ll \mu_t^x$ follows by symmetry.

  The consequence that $\mu_t$ is absolutely continuous with respect
  to left Haar measure is a standard fact about locally compact
  groups; see for instance \cite[Ch.~7, \S1.9, Proposition
    11]{bourbaki-integration-ii}.
\end{proof}

By the results in \cite{BB}, the reverse log Sobolev inequality holds
in sub-Riemannian manifolds satisfying a generalized
curvature-dimension inequality of the type introduced in
\cite{BaudoinGarofalo}.  It was shown in \cite{BaudoinGarofalo} that
such inequalities hold for step two Carnot groups and the
three-dimensional model groups $\mathbb{SU}(2)$ and $\mathbb{SL}(2)$,
and in \cite{BaudoinCecil} for three-dimensional solvable groups.

\subsection{Abstract Wiener space}\label{aws-sec}

The phenomenon of quasi-invariance is more interesting in groups that
are not locally compact, such as infinite dimensional vector spaces or
Lie groups.  Here, the smoothness of a measure cannot be described in
terms of absolute continuity to Haar measure, since Haar measure does
not exist, and so quasi-invariance provides a more ``intrinsic''
notion of regularity.

In this subsection, we consider the very classical example of abstract
Wiener space.  As this and similar infinite-dimensional models do not
fit exactly into the setting defined in Section \ref{setup-sec}, we
shall briefly discuss how to adapt the results of Sections
\ref{HK-sec} and \ref{renyi-sec} in this case, as a prototype for
later examples.  We give basic definitions here to fix notation; for
further background on abstract Wiener space and Gaussian measures on
infinite-dimensional spaces, we refer to \cite{bogachev-gaussian-book,
  kuo-gaussian-book}.

An \emph{abstract Wiener space} consists of a real separable Banach
space $W$ equipped with a centered non-degenerate Gaussian Borel
measure $\mu$.  We denote by $H \subset W$ the associated dense
\emph{Cameron--Martin space}, into which the continuous dual $W^*$ is
naturally embedded.  A smooth \emph{cylinder function} is a function
$F : W \to \mathbb{R}$ of the form $F(x) = \varphi(f_1(x), \dots,
f_n(x))$ for some $n$, where $\varphi \in C^\infty_b(\mathbb{R}^n)$ is
a smooth function with all partial derivatives bounded, and $f_1,
\dots, f_n \in W^* \subset H$; unless otherwise specified, we assume
without loss of generality that $f_1, \dots, f_n$ are orthonormal in
$H$.  We let $Cyl(W)$ denote the space of all such functions; this
will be used in place of $\lip_b(W)$ as a space of test functions.  It
is a standard fact that $Cyl(W)$ is dense in $L^p(\mu)$ for $1 \le p <
\infty$.

The
\emph{Malliavin gradient} $D F : W \to H$ of a cylinder function
is defined by $(D F)(x) = \sum_{i=1}^n (\partial_i
\varphi)(f_1(x), \dots, f_n(x)) f_i$, so that when the $f_i$ are
orthonormal in $H$ we have
\begin{equation*}
  \|DF(x)\|_H^2 = \sum_{i=1}^n |\partial_i \varphi(f_1(x), \dots,
  f_n(x))|^2 = |\nabla \varphi(f_1(x), \dots, f_n(x))|^2.
\end{equation*}
Note that $\|DF\|_H$ is not a strong upper gradient on $W$ with
respect to the distance induced by its norm $\|\cdot\|_W$.

The heat semigroup $P_t$ on $W$ is the convolution semigroup induced
by the rescaled measure $\mu$, namely $P_t F(x) = \int
F(x+\sqrt{t}y)\,\mu(\d y)$.  When $F$ is a cylinder function $F(x) = \varphi(f_1(x), \dots,
f_n(x))$, we have $P_tF(x) = p_t
\varphi(f_1(x), \dots, f_n(x))$ where $p_t$ is the standard heat
semigroup on $\mathbb{R}^n$; in particular, $P_t F$ is again a
cylinder function.

We recall that $p_t$ satisfies the
reverse Poincar\'e inequality
\begin{equation}
  |\nabla p_t \varphi|^2 \le \frac{1}{t} (p_t \varphi^2 - (p_t
  \varphi)^2), \qquad \varphi \in C^\infty_b(\mathbb{R}^n)
\end{equation}
and the reverse logarithmic Sobolev inequality
\begin{equation}
  p_t \varphi |\nabla \ln p_t \varphi|^2 \le \frac{2}{t}(p_t(\varphi \ln \varphi)
  - p_t \varphi \ln p_t \varphi), \qquad \varphi \in
  C^\infty_b(\mathbb{R}^n), \varphi > 0.
\end{equation}
These follow, for instance, by standard $\Gamma$-calculus from the
elementary commutation $\nabla p_t \varphi = p_t \nabla \varphi$.
See for instance \cite[Proposition 3.3]{bakry-tata}, taking $\rho = 0$.
Note that the constants in these inequalities are
dimension-independent.  As such, evaluating at $(f_1(x),
\dots, f_n(x))$, $x \in W$, we obtain the corresponding inequalities
for $P_t$ on $(W,\mu)$:
\begin{align}
  \|D P_t F\|^2_H &\le \frac{1}{t} (P_t F^2 - (P_tF)^2), \qquad F \in
  Cyl(W) \label{aws-rev-poincare} \\
  P_t F \|D \ln P_t F\|^2_H &\le \frac{2}{t}(P_t(F \ln F)
  - P_t F \ln P_t F), \qquad F \in
  Cyl(W), F > 0.
\end{align}

We modify Definitions \ref{W-def} and \ref{T-def} and by taking our class of test functions
to be smooth cylinder functions of space and time, e.g. functions
$F_s : [0,1] \times W \to \mathbb{R}$ of the form $F_s=\varphi(s,
f_1(x), \dots, f_n(x))$, $\varphi \in C^\infty_b([0,1] \times
\mathbb{R}^n)$.  Let $Cyl([0,1] \times W)$ denote the space of such
functions.  Then we redefine
\begin{align*}
  \mathcal{A}_{a,b} &= \left\{
  F_s \in Cyl([0,1] \times W) : \partial_s F_s + a \|DF_s\|_H^2 + b
  F_s^2 \le 0\right\} \\
  \mathcal{E}_{a,b} &= \left\{
  F_s \in Cyl([0,1] \times W) : F > 0, \partial_s F_s + a F_s \|D \ln F_s\|_H^2 + b
  F_s^2 \le 0\right\} 
\end{align*}
and define $W_{a,b}$, $T_{a,b}$ accordingly on
$\mathcal{P}(W)$.  We have $W_{0,b}$ and $T_{0,b}$ related to
Hellinger and R\'enyi divergences in the same way as before.  Moreover
we can follow the proof of the upper bound in Proposition
\ref{T-point-mass-improved}, taking $\gamma(s) = s x_1 +(1-s)x_0$ and
noting $\left|\frac{d}{ds} F(\gamma(s))\right| \le \|DF(\gamma(s))\|_H
\|x_1 - x_0\|_H$, to conclude
\begin{equation}\label{aws-T-point-mass}
  T_{a,b}(\delta_{x_0}, \delta_{x_1}) \le C_b \exp\left(\frac{1}{4ab}\|x_0 - x_1\|_H^2\right).
\end{equation}

Now Theorem \ref{entropic-equivalence} allows us to recover the classical
Cameron--Martin quasi-in\-var\-iance theorem \cite{cameron-martin-44}.  For $t > 0$,
let $\mu_t = \mu(t^{-1/2} \,\cdot\,) = \delta_0 P_t$ be the rescaling of
the Gaussian measure $\mu$, and for $h \in H$ let $\mu_t^h =
\mu(t^{-1/2}(\cdot-h)) = \delta_h P_t$ be its translation by
$h$.  We then obtain:

\begin{proposition}[Cameron--Martin theorem]\label{cm-rev-lsi}
  For all $t > 0$ and $h \in H$, the measures $\mu_t, \mu_t^h$ are
  mutually absolutely continuous.
\end{proposition}

\begin{proof}
  The logic is the same as in the proof of Proposition \ref{lie-quasi}.  Since the
  reverse logarithmic Sobolev inequality holds, Theorem
  \ref{entropic-equivalence} and \eqref{aws-T-point-mass} imply that
  for any $\kappa > 0$, we have
  \begin{equation}\label{gaussian-T-estimate}
    T_{0,2\kappa /t}(\mu_t, \mu_t^h) \le T_{1,2\kappa /t}(\delta_0, \delta_h) \le
    C_{2\kappa/t} \exp\left(\frac{t}{8\kappa^2} \|h\|_H^2\right) < \infty 
  \end{equation}
  Thus by Lemma \ref{T-abs-cont} we have $\mu_t^h \ll \mu_t$, and the
  reverse statement $\mu_t \ll \mu_t^h$ follows by symmetry.
\end{proof}

We also obtain a quantitative estimate on the $L^p$ norm of the
density $d\mu_t^h/d\mu_t$, which is perhaps most convenient to
consider in the form of \eqref{weak-ihi}:
\begin{equation}\label{gaussian-renyi-estimate}
  \left(\int_W \left(\frac{\d\mu_t}{d\mu_t^h}\right)^{1/(p-1)}
  \d\mu_t\right)^{p-1} \le \exp\left(\frac{p-1}{(\log p)^2} \frac{\|h\|_H^2}{2t}\right).
\end{equation}
As the left side is known to exactly equal $\exp\left(\frac{p}{p-1}
\frac{\|h\|_H^2}{2t}\right)$ (so that \eqref{ihi} is sharp),
\eqref{gaussian-renyi-estimate} becomes sharp as $p \to 1$.

One may also apply the reverse logarithmic Sobolev inequality for the
Ornstein--Uhlenbeck $Q_s$, which is the symmetric Markov semigroup on $L^2(\mu)$
generated by the Dirichlet form $\mathcal{E}(F,F) = \int_W
\|DF\|_H^2\,d\mu$.  It satisfies 
\begin{equation}
    Q_s F \|D \ln Q_s F\|^2_H \le \frac{2}{e^{2s}-1}(Q_s(F \ln F)
  - Q_s F \ln Q_s F).
\end{equation}
See for instance \cite[Section 3]{bakry-tata}, noting that the carr\'e du champ
of $Q_s$ is $\Gamma(F,F) = \|DF\|_H^2$, without a factor of
$\frac{1}{2}$.  Carrying out the above computations with $Q_s$ and
noting that $\delta_h Q_s = \mu_{1-e^{-2s}}^{e^{-s}h}$, one obtains
exactly the same results for $t < 1$.

Finally, we remark that the Cameron--Martin quasi-invariance theorem can also be obtained
  using the Hellinger--Kantorovich contraction property of Theorem \ref{equi-func}.  Indeed, the reverse
  Poincar\'e inequality  \eqref{aws-rev-poincare} for $P_t$ implies
  \begin{equation}\label{aws-hellinger-bound}
    \mathsf{He}_2(\mu_t, \mu_t^h)^2  \le \frac{1}{4t} \|h\|_{H}^2
  \end{equation}
  since the Kantorovich--Wasserstein distance between point masses in
  this setting corresponds to the Cameron--Martin distance; this can
  be checked directly from the dynamic dual definition as in
  Proposition \ref{W-Dirac-distance-prop}.  Unfortunately,
  \eqref{aws-hellinger-bound} has no content unless $\frac{1}{4t}
  \|h\|_H^2 < 2$, so to work around this, choose an integer $n$ so
  large that $n^{-2} \frac{1}{4t}  \|h\|^2_H < 2$.  Applying
  \eqref{aws-hellinger-bound} with $h/n$ in place of $h$, we conclude
  that $\mathsf{He}_2^2(\mu_t, \mu_t^{h/n}) < 2$ and in particular
  that $\mu_t, \mu_t^{h/n}$ are not mutually singular.  By the
  Feldman--H\'ajek dichotomy theorem for Gaussian measures \cite{feldman,feldman-correction,hajek-dichotomy,brody-dichotomy}, they must
  therefore be mutually absolutely continuous, which we denote by
  $\mu_t \sim \mu_t^{h/n}$.  Repeating this argument $n$ times, we
  have $\mu_t \sim \mu_t^{h/n} \sim \mu_t^{2h/n} \sim \dots \sim
  \mu_t^h$, and since $\sim$ is an equivalence relation, we have
  $\mu_t \sim \mu_t^h$ as desired.

  Although this argument uses only the reverse Poincar\'e inequality,
  which is \emph{a priori} weaker than the reverse logarithmic Sobolev
  inequality used in Proposition \ref{cm-rev-lsi}, the conclusion is
  also weaker as it does not yield any quantitative information about
  the distance between the measures $\mu_t, \mu_t^h$.

  We note that some proofs of the Feldman--H\'ajek dichotomy theorem,
  including Feldman's original proof \cite{feldman,
    feldman-correction}, make use of the Cameron--Martin
  quasi-invariance theorem, which would seem to make the above
  argument circular.  However, it is possible to prove the dichotomy
  theorem directly, without assuming quasi-invariance---see for
  example \cite{brody-dichotomy}---and this breaks the cycle.

\subsection{Infinite dimensional Heisenberg-like groups}

The ideas of the previous two subsections come together in the study
of infinite-dimensional groups where the semigroup in question is not
elliptic.  In \cite{bgm}, the authors considered infinite-dimensional
Heisenberg-like groups, introduced in \cite{DriverGordina}, with their
hypoelliptic heat kernels and corresponding heat semigroups.  These
groups carry a natural sub-Riemannian geometry analogous to the
Heisenberg group and other Carnot groups of step two.   They use
generalized curvature-dimension inequalities to show that these spaces
satisfy a reverse logarithmic Sobolev inequality.  From this, they
derive a Wang-type Harnack inequality, and use this to show
quasi-invariance of the heat kernel measure under the group
translation.  In this section, we show that as in the case of Gaussian
measures, transport inequalities provide an alternate route from
reverse log Sobolev to quasi-invariance in this setting.  We only
sketch the argument here, as the details are closely analogous to
those for the Gaussian case.

We follow the notation of \cite{bgm} and refer the reader there for
complete definitions, background, and further references.  Let
$(W,H,\mu)$ be an abstract Wiener space and $\mathbf{C}$ a
finite-dimensional inner product space.  Suppose that $\mathfrak{g} =
W \times \mathbf{C}$ is equipped with a continuous Lie bracket
$[\cdot, \cdot]$ satisfying $[W,W]=\mathbf{C}$ and $[\mathfrak{g},
  C]=0$.  The corresponding Banach Lie group $\mathbb{G}$ is given by $\mathbb{G} = W
\times \mathbf{C}$ equipped with the nonabelian group operation $g_1
\cdot g_2 = g_1 + g_2 + \frac{1}{2} [g_1, g_2]$ defined by the
Baker--Campbell--Hausdorff formula.  Then $\mathfrak{g}_{CM} = H
\times \mathbf{C}$ is a dense Lie subalgebra of $\mathfrak{g}$, called
the Cameron--Martin Lie subalgebra, and likewise $\mathbb{G}_{CM} = H \times
\mathbf{C} \subset \mathbb{G}$ is a dense subgroup of $\mathbb{G}$.

If $B_t$ is a standard Brownian motion on $(W,\mu)$, we may define a
left-invariant Brownian motion $g_t$ on $\mathbb{G}$ by the formula $g_t =
\left(B_t, \frac{1}{2}\int_0^t [B_s,\d B_s] \right)$.  Let $\nu_t
= \operatorname{Law}(g_{2t})$ be the heat kernel measure induced by
$g_t$.  By analogy with the finite-dimensional Heisenberg group, one
expects the measure $\nu_t$ to be ``smooth'' in some sense.  One
cannot express this smoothness in terms of a density with respect to
Lebesgue or Haar measure because the latter do not exist in infinite
dimensions, but another reasonable notion of smoothness would be for
$\nu_t$ to be quasi-invariant under left translation by elements of
the Cameron--Martin subgroup $\mathbb{G}_{CM}$.  The main result of \cite{bgm}
is that this is in fact the case.  (We also mention \cite{dem} where
the same statement was shown through different means, by producing a
density of $\nu_t$ with respect to the measure $\mu \times m$, where
$\mu$ is the Gaussian measure on $W$ and $m$ is Lebesgue measure on $\mathbf{C}$.)

It is shown in \cite{bgm} that the group $\mathbb{G}$ can be
approximated by finite-dimensional projection groups $\mathbb{G}_P$,
each of which is a nilpotent Lie group of step $2$.  This leads to a
notion of smooth cylinder functions $F : \mathbb{G} \to \mathbb{R}$
which can be differentiated in directions $X \in \mathfrak{g}_{CM}$,
and thus a horizontal gradient $\nabla_H F : \mathbb{G} \to H$ can be
defined for such functions.  If $\gamma : [0,1] \to \mathbb{G}_{CM}$
is an absolutely continuous horizontal path, then its derivative
$\gamma'$ can be identified as a curve in $H$, and we have the chain
rule $\frac{d}{ds} F(\gamma(s)) = \langle \nabla_H F(\gamma(s)),
\gamma'(s) \rangle_H$.  Moreover, $\mathbb{G}_{CM}$ is a length space
with respect to the horizontal distance $d_{CM}$, and so the estimates
on $W_{a,b}(\delta_0, \delta_g), T_{a,b}(\delta_0, \delta_g)$ from
Propositions \ref{W-Dirac-distance-prop} and
\ref{T-point-mass-improved} go through for $g \in \mathbb{G}_{CM}$,
with $d = d_{CM}$.

Now \cite[Proposition 4.8]{bgm} shows, by means of generalized
curvature-dimension inequalities as introduced in \cite{BaudoinGarofalo}, that each projection
group $\mathbb{G}_P$ satisfies a reverse logarithmic Sobolev inequality, with a
uniform constant of the form $C/t$ where $C$ depends only on the
structure of $\mathbb{G}$, and not on the projection.  This can be restated as the
following reverse logarithmic Sobolev inequality for cylinder
functions on $\mathbb{G}$: 
\begin{equation}\label{inf-heis-rev-lsi}
  P_t F \|\nabla \ln P_t F\| \le \frac{C}{t} (P_t(F \ln F)-P_t f \ln
  P_t f)
\end{equation}
and so as in Proposition \ref{cm-rev-lsi} above, we recover a version
of the main quasi-invariance result of \cite{bgm} and \cite{dem}:
\begin{proposition}
  For each $t>0$, the heat kernel measure $\nu_t$ on $\mathbb{G}$ is
  quasi-invariant under left translation by elements of $\mathbb{G}_{CM}$.
\end{proposition}
Moreover, the bounds on $T_{0,b}(\nu_t, \nu_t^g)$ in \eqref{eti}
yield $L^q$ bounds on the Radon--Nikodym derivative $\d\nu_t^g/\d\nu$,
as in the proof of Theorem \ref{entropic-equivalence}, which are
asymptotically equivalent to the integrated Harnack inequalities of
\cite[Section 5.2]{bgm} as $p \to 1$ and $q \to \infty$.

\section*{Acknowledgments}
The authors are grateful for helpful
discussions with Maria Gordina, Martin Hairer, Ronan Herry, Kazumasa Kuwada,
Xue-Mei Li, and Giuseppe Savar\'e.  We also thank the anonymous
referee for their careful reading and useful suggestions.  This
article was completed during a sabbatical visit by author N.~Eldredge
to the Department of Mathematics at the University of Connecticut; he
would like to thank the Department and especially Maria Gordina for
their hospitality, especially in view of the difficult circumstances
created by the COVID-19 pandemic.

\bibliographystyle{plainnat}
 \bibliography{Duality_Refs}

\end{document}